\documentclass{amsart}
\usepackage{amsmath,amssymb,amsthm,mathtools,mathdots,nicefrac,tikz,enumitem,url,graphicx}

\usepackage{tikz-3dplot}

\usepackage[a4paper,top=3cm,bottom=2cm,left=3cm,right=3cm,marginparwidth=1.75cm]{geometry}
\usepackage{pst-platon}
\usepackage{manfnt}
\usepackage{appendix}

\usepackage{algorithm}
\usepackage[noend]{algpseudocode}
\usepackage{mdframed}

\theoremstyle{definition} 
\newtheorem{theorem}{Theorem}

\newtheorem{lemma}{Lemma}
\newtheorem{proposition}{Proposition}

\newtheorem{chipfiringgame}{Chip-Firing Game}
\newtheorem{exploringfurther}{Exploring Further}

\newcommand{\sn}{\mathrm{sn}}

\DeclareMathOperator{\outdeg}{outdeg}

\DeclareMathOperator{\val}{val}
\DeclareMathOperator{\tw}{tw}
\DeclareMathOperator{\gon}{gon}

\DeclareMathOperator{\scw}{scw}

\DeclareMathOperator{\tetrahedron}{\includegraphics[height=0.1383in]{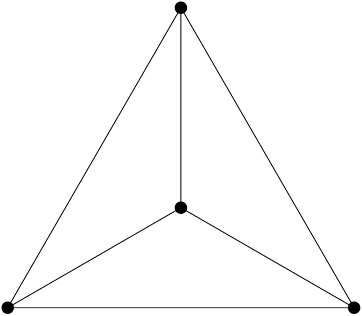}_T}
\DeclareMathOperator{\octahedron}{\includegraphics[height=0.1383in]{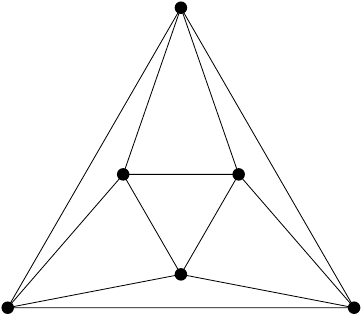}_O}
\DeclareMathOperator{\cube}{\includegraphics[height=0.1383in]{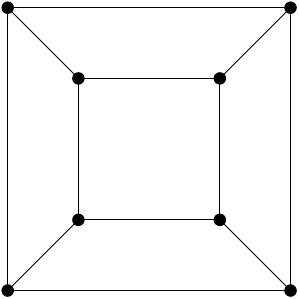}_C}
\DeclareMathOperator{\dodecahedron}{\includegraphics[height=0.1383in]{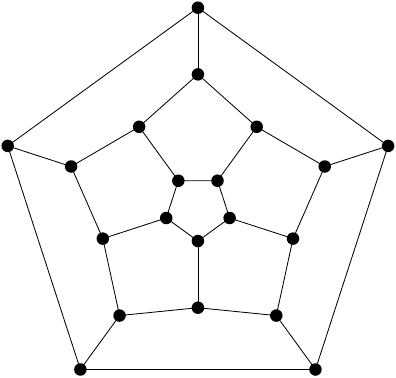}_D}
\DeclareMathOperator{\icosahedron}{\includegraphics[height=0.1383in]{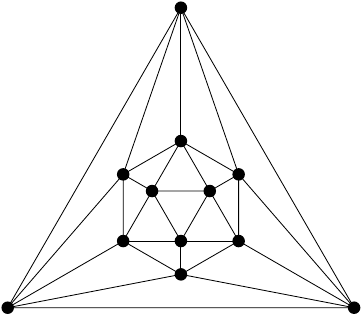}_I}

\usepackage{mathrsfs}

\title{Chip-firing on the Platonic solids:  a primer for studying graph gonality}
\author{Marchelle Beougher, Kexin Ding, Max Everett, Robin Huang, Chan Lee, Ralph Morrison, and Ben Weber}
\date{}

\begin{document}

\maketitle

\begin{abstract}
    This paper provides a friendly introduction to chip-firing games and graph gonality.  We use graphs coming from the five Platonic solids to illustrate different tools and techniques for studying these games, including independent sets, treewidth, scramble number, and Dhar's burning algorithm.  In addition to showcasing some previously known results, we present the first proofs that the dodecahedron graph has gonality $6$, and that the icosahedron graph has gonality~$9$.
\end{abstract}

\section{Graphs and chip-firing}

Take a piece of paper, and draw some dots.  Then, draw some curves connecting some dots to one another--those curves can be straight or not, and it's fine if they cross.  What you end up with is called a \emph{graph}.  To put some names to things, each dot is called a \emph{vertex}, and each curve connecting two vertices is called an \emph{edge}.  Usually we'll write $G$  for the graph, $V$ or $V(G)$ for the set of vertices, and $E$ or $E(G)$ for the set of edges.  The \emph{valence} of a vertex $v\in V(G)$, written $\val(v)$, is the number of edges with $v$ as an endpoint.  Two graphs are drawn in Figure \ref{figure:two_graphs}, each with $6$ vertices and $7$ edges.  The first has vertices with valences $3$, $3$, $3$, $3$, $1$, and $1$, and the second $3$, $3$, $2$, $2$, $2$, and $2$.

\begin{figure}[hbt]
    \centering
    \includegraphics{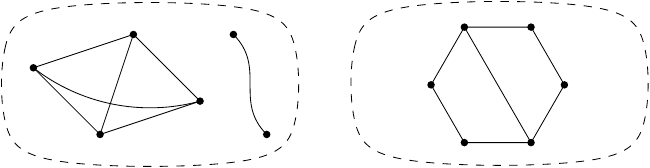}
    \caption{Two graphs, each with six vertices and seven edges}
    \label{figure:two_graphs}
\end{figure}

There are lots of rules we'll have to set for drawing our graphs.  Can they be made up of multiple pieces, like the graph on the left?  Are we allowed to draw multiple edges between vertices, or an edge from  a vertex to itself?  Should we keep track of the directions of edges, so that one vertex is the start and the other vertex is the end?  And do we have to limit ourselves to finite graphs, or could we have \emph{infinitely many} vertices are edges? Depending on what paper you're reading or what graph theorist you're talking to, you'll find different answers to these questions!

For us, here will be our rules:  we'd like our graphs to be \emph{connected}, meaning intuitively they're made up of a single piece; more precisely, for any pair of vertices, we can travel from one to the other using some number of edges.  So the graph on the right in Figure \ref{figure:two_graphs} is fine, but not the graph on the left.  We're fine with multiple edges  connecting a pair of vertices, but not with a ``loop'' edge connecting a vertex to itself; some people call graphs of this form \emph{multigraphs}, or if loops are allowed, \emph{pseudographs}.  We're not assigning any direction to our edges, so our graphs are \emph{undirected}; think of each edge as a two-way street, with things being able to move in either direction.  Finally, to keep things a little easier, we'll assume the vertex and edge sets are both finite.

Now that we've set our rules, we can draw a graph and play a \emph{chip-firing game} on it.  On each vertex, place some number of poker chips.  We don't allow fractional chips, but negative numbers are allowed, perhaps representing debt.  Such a placement is called a \emph{chip-configuration} or a \emph{divisor}\footnote{It's probably not clear why we would  call this a ``divisor''; what does our set-up  have to do with dividing or divisibility?  We're borrowing the word from another field of math, called \emph{algebraic geometry}, which studies polynomial equations and solutions to them.  As explored in \cite{bn}, there's a surprisingly deep connection between algebraic geometry and chip-firing!} on the graph.  We represent these chips with integer labels on the vertices, usually leaving off a label of $0$ when the vertex has $0$ chips. We then move chips around on $G$ via \emph{chip-firing moves}.  To perform such a move, we pick a vertex of $G$, and have it donate chips to its neighbors, one along each edge touching the vertex.  You can see an example of a few chip-firing moves in Figure \ref{figure:chip_firing_moves}, where vertex \(v\) is fired followed by vertex \(w\).  You can check that if we'd fired vertex \(w\) first and vertex \(v\) second, we would've wound up with the same chip placement--it turns out that the order of firing moves doesn't matter!

\begin{figure}[hbt]
    \centering
    \includegraphics{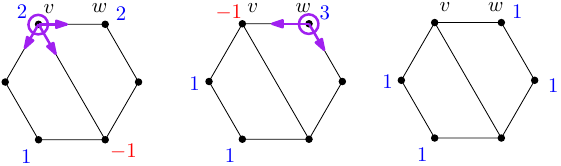}
    \caption{Three divisors related by chip-firing moves}
    \label{figure:chip_firing_moves}
\end{figure}

If we can perform a sequence of chip-firing moves to turn a divisor $D$ into another divisor $D'$, we say that $D'$ is \emph{equivalent} to $D$.  This choice of terminology leads to a natural question:  if $D'$ is equivalent to $D$, then is $D$ equivalent to $D'$?  In other words, can we ``reverse'' chip-firing moves It turns out the answer is yes!  To undo firing at a vertex $v$, simply fire all the other vertices (in any order).  The effect of \emph{all} these chip-firing moves is that every vertex fires once, so every chip that gets donated gets donated back, and we're back where we started\footnote{It turns out that chip-firing equivalence forms what's called an \emph{equivalence relation}.  Here we were discussing a property called symmetry, but it also satisfies reflexivity and transitivity.}. 

There are lots of solitaire games we can play using chip-firing moves.  For instance, given a divisor, how can we perform chip-firing moves to make it so every vertex is out of debt, as done in Figure \ref{figure:chip_firing_moves}?  Or if it's impossible, how do we know?

\begin{chipfiringgame}[The Dollar Game] Given a placement of chips $D$ on a graph \(G\) can we perform chip-firing moves to eliminate all debt in the graph?
\end{chipfiringgame}

The sequence of chip-firing moves in Figure \ref{figure:chip_firing_moves} can be viewed as an instance of winning the Dollar Game, starting from the leftmost chip placement, since the final chip placement has no debt.

There are  times when we can quickly come up with an answer to the Dollar Game.  For instance, if there is no debt to start, we have already won!  Or, if the total number of chips is negative, then there's no way to win--the total number of chips cannot be changed by chip-firing moves, so there will always be debt somewhere.  It turns out there are lots of different methods to win this game (or to prove it's unwinnable) in general; you can read about a few of them in \cite[Chapter 3]{sandpiles}.  We'll present one strategy, sometimes called \emph{Dhar's burning algorithm} \cite{dhar}.

\begin{enumerate}
    \item First, pick a vertex \(q\) in the graph, and  get every vertex besides \(q\) out of debt.  (The idea here is to fire \(q\) lots and lots of times to generate chips elsewhere, and then to move them around the graph to eliminate any other debt.  See if you can argue this is always possible!)  If $q$ is out of debt, we are done; otherwise move on to step (2).

    \item  Start a ``fire'' at \(q\), so that the vertex $q$ is burning.  The fire spreads through the graph according to the following rules:  
    \begin{itemize}
        \item Edges are very flammable, so any edge touching a burning vertex will burn.
        \item  Vertices can protect themselves by using their chips as ``firefighters'', each of which can fight off one burning edge.  In particular, as long as a vertex has as many or more chips as there are burning edges touching it, that vertex is safe.  But as soon as there are more burning edges than chips, that vertex burns.
    \end{itemize}
    \item If the whole graph burns, stop the algorithm. If the burning process terminates and some vertices remain unburned, chip-fire every one of those vertices.  If debt is eliminated on \(q\), then we're done!  If not, then we repeat the process from step (2).
\end{enumerate}

It turns out that one of two things will happen.  Either:  debt will be eliminated on \(q\), and we've won the Dollar Game; or eventually the whole graph will burn while debt is still on \(q\), meaning the Dollar Game is unwinnable. (That shouldn't be obvious--it takes a proof to conclude that!)

\begin{figure}[hbt]
    \centering
    \includegraphics{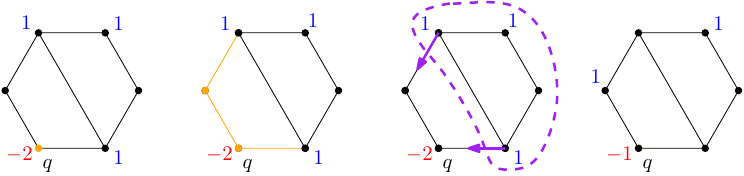}
    \caption{An example of Dhar's burning algorithm}
    \label{figure:dhars_burning}
\end{figure}

An illustration of Dhar's burning algorithm is presented in Figure \ref{figure:dhars_burning}.  All debt is already on the vertex $q$, so we begin the burning process.  The edges incident to $q$ burn, as does the one incident unchipped vertex and one more edge, but then the burning process stops there since any unburned vertex incident to burning edges has at least as many chips as there are burning edges.  That set of vertices is then fired to give a new divisor.  To complete the algorithm, we would run the burning process again since $q$ is still in debt; this time, the whole graph would burn.  Thus the Dollar Game starting with this divisor cannot be won, no matter how many chip-firing moves we perform!

Step (4) of our process highlights a very important idea in chip-firing games:  \emph{set-firing moves}.  These occur when rather than saying ``fire this vertex, then that vertex, then that vertex'', we instead say ``let $S$ be a collection of vertices, and fire every vertex in $S$''.  Really that means to fire the vertices in $S$ in any order, since the order of firing won't matter.  The key thing to observe is that if two vertices connected by an edge are both fired, then a chip moves from each vertex to the other, cancelling out.  So, the only net movement of chips is along edges connecting a vertex in $S$ to a vertex outside of $S$.  This is actually why no new debt is introduced in step (3):  chips only move from unburned vertices to burned vertices, and the way our burning process works there will always be enough chips to accomplish this!

The game we'll be focusing on in this article is one for two players, and is called the \emph{Gonality Game}.  

\begin{chipfiringgame}[The Gonality Game] Let \(G\) be a graph, and \(N\) a nonnegative integer.

\begin{enumerate}
    \item Player A places $N$ chips on the vertices of $G$.
    \item  Player B adds $-1$ chips, possibly putting a vertex in debt.
    \item  Player A makes any number of chip-firing moves on the resulting divisor.
\end{enumerate}
If Player A can remove all debt, then they win; if they can't, then Player B wins. Note that by the time we get to phase (3), Player A is just playing an instance of the Dollar Game.
\end{chipfiringgame}

\begin{figure}[hbt]
    \centering
    \includegraphics[scale=0.8]{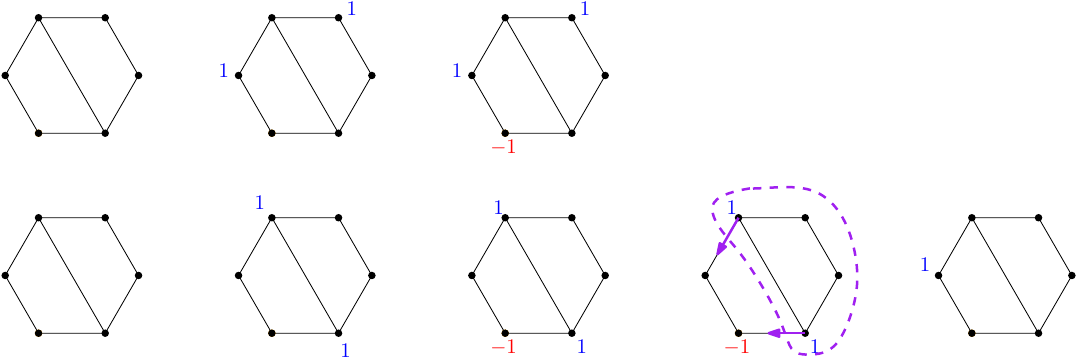}
    \caption{Two examples of the Gonality Game, one where Player B wins and one where Player A wins}
    \label{figure:gonality_game_examples}
\end{figure}

Two run-throughs of the Gonality Game are illustrated in Figure \ref{figure:gonality_game_examples}, one on the top and one on the bottom.  In both instances, Player A places $2$ chips, and Player B places $-1$ chips. By Dhar's burning algorithm, Player A cannot eliminate debt in the top game, so Player B wins that Gonality Game.   In the bottom game, Player A can succeed in eliminate debt as illustrated, so Player A wins that Gonality Game.

 Let's assume that Player A is very smart about their initial placement, and Player B is very smart about where they put the $-1$.  For instance, after Player A places their chips, Player B could think through every possible vertex to place their \(-1\) on, and use Dhar's  Burning Algorithm on each placement to see if any of the resulting Dollar Games are unwinnable;  if any unwinnable placements exist, Player B would choose that.  And stepping further back, Player A could pre-empt Player B by thinking through \emph{every} placement of \(N\) chips, followed by \emph{every} subsequent placement of \(-1\) chips, and check if any placement of \(N\) chips remains winnable regardless of where the \(-1\) is placed.  (It's worth noting that that's a lot of placements for Player A to think through, assuming $N$ is large!)
 
 The main factor that will determine the outcome of the game is then the number $N$ of chips Player A gets to play.  In one extreme case, if $N=0$, then Player A will lose, since Player B can play anywhere and debt will never be eliminated (since the total number of chips is negative).
Towards the other extreme, if $N$ is equal to the number of vertices, then Player A can place a chip on every vertex and win automatically, since Player B's placement of $-1$ chips does not introduce debt.  Somewhere in the middle then, there must be a changeover:  a value of $N$ such that Player A has a winning strategy with $N$ chips, but not with $N-1$ chips.  This minimum possible $N$ is called the \emph{gonality} of $G$, written $\textrm{gon}(G)$ (or in some papers, $\textrm{dgon}(G)$, for \emph{divisorial gonality}.)

For our first computation of the gonality of a graph, let's use the graph from Figure \ref{figure:gonality_game_examples}.  Player A managed to win the second game, and in fact would have won regardless of where Player B had placed their $-1$.  That means the divisor of degree $2$ played by Player A wins the Gonality Game, and so $\gon(G)\leq 2$.  But could Player A win with just one chip, rather than two?  The answer is no--running Dhar's burning algorithm from any other vertex with $-1$ debt will burn the whole graph regardless of where a single chip is placed, so Player B will always be able to counter Player A's $1$-chip strategy.  Thus we know that $\gon(G)=2$.

It turns out that we already have the pieces to build an algorithm to compute the gonality of a graph!\footnote{This algorithm, or a slightly streamlined version of it, has been implemented on this website:  \url{https://chipfiringinterface.web.app/}.  Try it out!}

\begin{enumerate}
    \item Let \(N=1\).
    \item For every placement of \(N\) chips, check whether that placement wins the Dollar Game for every additional placement of \(-1\) chips.
    \item  If any placement from (2) was successful, then \(\gon(G)=N\).  Otherwise, increase \(N\) by \(1\), and go back to step \(2\).
\end{enumerate}
Sadly, this algorithm is very inefficient: since \(N\) might have to get very large, we are looking at lots and lots\footnote{It turns out exponentially many!} of possible chip placements.  Indeed, people have proved that computing the gonality of a graph is ``NP-hard'', meaning it's at least as hard as some problems that most people don't believe can be solved by fast algorithms (see \cite{gijswijt2019computing} for the first proof of NP-hardness, or \cite{echavarria2021scramble} for a slightly simpler proof).

So if someone gives us a graph $G$, how should we compute its gonality, without relying on a slow algorithm?  Well, suppose we suspect $\gon(G)=N$.  We can prove $\gon(G)\leq N$ by finding a chip placement with $N$ chips, and arguing it can eliminate added debt anywhere.  Simple enough to say, but sometimes hard in practice, so it'd be great to have general strategies for finding nice chip placements!  Then we have to prove $\gon(G)\geq N$, which is a bit harder.  The most straightforward approach would be to argue that no placement of \(N-1\) chips works, which might be manageable or might be too brute force; it'd be great to have other tools to lower bound gonality.



\begin{figure}[hbt]
    \centering
    \includegraphics[scale=0.49]{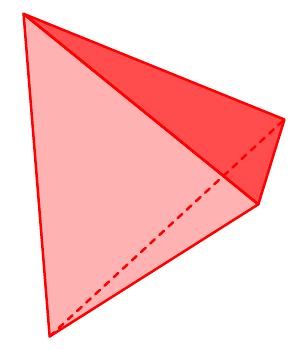}\quad\includegraphics[scale=0.69]{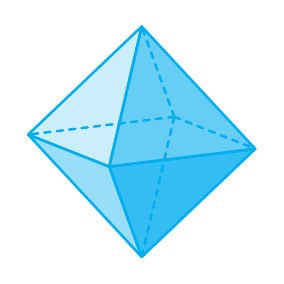}\quad\includegraphics[scale=0.49]{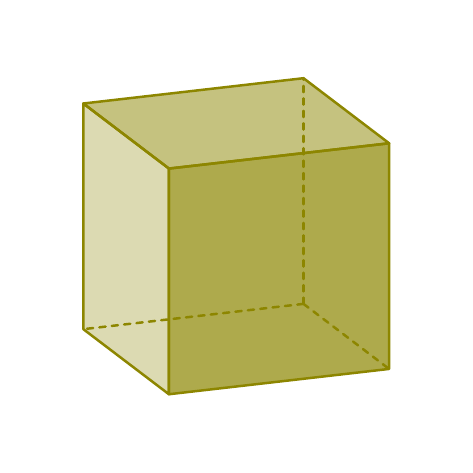}\quad\includegraphics[scale=0.39]{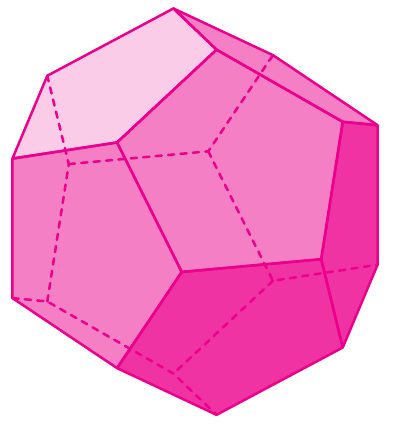}\quad\includegraphics[scale=0.39]{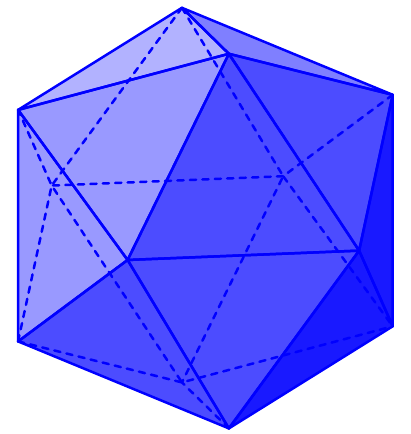}
    \caption{The five Platonic solids:  tetrahedron, octahedron, cube, dodecahedron, icosahedron}
    \label{figure:platonic_solids}
\end{figure}

In this article we showcase some results and techniques that let us achieve upper and lower bounds on gonality.  We will work with five graphs, namely those coming from the \emph{Platonic solids}.  These are the five  three-dimensional polyhedra whose faces are identical regular polygons, such that every corner of the solid looks like every other corner. These geometric figures are illustrated in Figure \ref{figure:platonic_solids}.  We can turn these into graphs by considering only their vertices and edges, as illustrated in Figure \ref{figure:graphs_of_platonic_solids}.

\begin{figure}[hbt]
    \centering
    \includegraphics[scale=0.4]{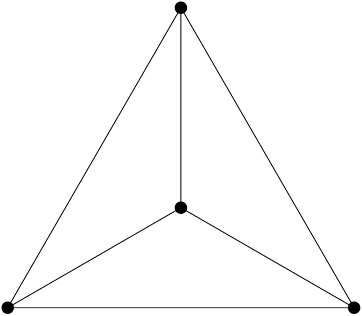} \quad   \includegraphics[scale=0.4]{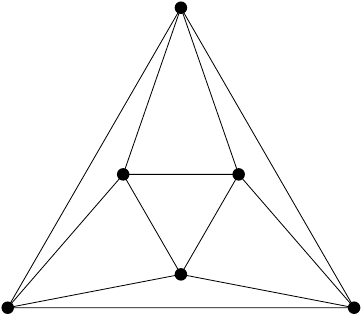} \quad   \includegraphics[scale=0.4]{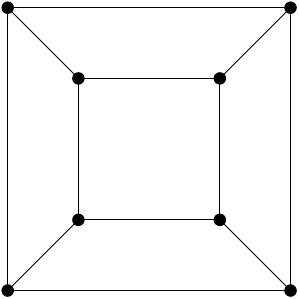} \quad   \includegraphics[scale=0.4]{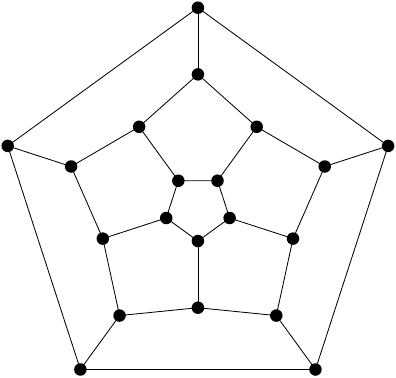} \quad   \includegraphics[scale=0.4]{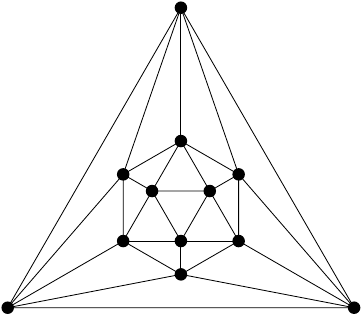} 
    \caption{The graphs of the Platonic solids}
    \label{figure:graphs_of_platonic_solids}
\end{figure}

Along the way to finding the gonalities of these Platonic graphs, we will think about other  families of graphs into which these fall, and showcase what we know (and what we don't yet know!) about their chip-firing properties.  Throughout, we'll denote each Platonic graph with a miniature picture of the graph, along with a subscript of the first letter of that solid in case the picture is hard to make out.  So,
\[\tetrahedron,\octahedron,\cube,\dodecahedron,\icosahedron\]
denote the tetrahedron graph, the octahedron graph, the cube graph, the dodecahedron graph, and the icosahedron graph, respectively.

Here are some of the tools and ideas you'll see in each section.

\begin{itemize}
    \item The tetrahedron: a ``Dhargument'', or an argument based on Dhar's burning algorithm;  complete graphs; parking functions.
    \item The octahedron: independent sets; treewidth; minimum degree; complete multipartite graphs.
    \item The cube: product graphs; scramble number; hypercube graphs.
    \item The dodecahedron; a more involved scramble number proof.
    \item The icosahedron: a more invovled Dhargurment; the Archimedean solids and higher dimensional Platonic solids.
\end{itemize}

\noindent \textbf{Acknowledgements.} The authors were supported by NSF Grants DMS-1659037, DMS-2011743, and DMS-2241623.  The Platonic solids in Figure \ref{figure:platonic_solids} were made using TeX code by Sebastian Tronto, available online \cite{platonic_solids_github}, and the truncated icosahedron in Figure \ref{figure:truncated_icosahedron} was made using GeoGebra\footnote{\url{https://www.geoge
bra.org/}}.

\section{The Tetrahedron}

Our first Platonic solid is the tetrahedron.  The boundary of this solid consists of four triangular faces, meeting along a total of six edges and at four vertices.  The underlying graph $\tetrahedron$ thus consists of $4$ vertices, every pair of which is connected by an edge.  This makes it a \emph{complete graph}, namely the complete graph on $4$ vertices $K_4$.  In general, the complete graph on $n$ vertices is written $K_n$.  You can see the first few complete graphs illustrated in Figure \ref{figure:complete_graphs}.

\begin{figure}[hbt]
    \centering
    \includegraphics{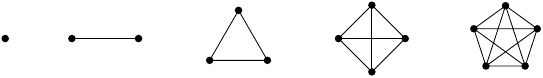}
    \caption{The complete graphs \(K_1\) through \(K_5\).  The tetrahedron graph is \(K_4\).}
    \label{figure:complete_graphs}
\end{figure}

Complete graphs are a very nice family of graphs to study if we want to compute gonality.  As it turns out, the gonality of $K_n$ is equal to $n-1$. To win the Gonality Game, one strategy is to put \(n-1\) chips on a single vertex \(v\).  Then, no matter where Player B places their \(-1\) chip (besides \(v\)), you can chip-fire \(v\) to move a chip to every vertex besides \(v\), thus winning the game.   To show that we can't do any better than $n-1$ chips, let's try using Dhar's burning algorithm.  

If Player A places fewer than $n-1$ chips, there will be at least two vertices with zero chips.  Once Player B places $-1$ chips on one of these unchipped vertices, run Dhar's burning algorithm.  If the whole graph burns, then Player B wins. 

Let's suppose for a moment the whole graph doesn't burn; we'll try to reach a contradiction. We then end up with some positive number of unburned vertices, let's say $k$ of them, meaning that \(n-k\) vertices are burned. Since the fire has stabilized, and every vertex is connected to every other vertex, we know each unburned vertex has at least \(n-k\) chips to fight off the burning edges. That means there must be a total of at least $k\cdot (n-k)$ chips in our placement.  You can prove (using calculus, give it a try!) that for $k$ an integer between $0$ and $n$, we have $k\cdot (n-k)$ is always at least $n-1$, except when \(k=0\) and \(k=n\). But we know $k\geq 1$ and $k\leq n-1$, since at least some of the graph burned without the whole graph burning. That means we have at least \(n-1\) chips, a contradiction since Player A only placed $n-2$!  That means the whole graph must have burned, so Player B wins, and Player A can't win with fewer than \(n-1\) chips.

This means that the gonality of $K_n$ is at least $n-1$, and we already argued that it is at most $n-1$. Thus, the gonality of $K_n$ is equal to $n-1$, and the gonality of the tetrahedron is equal to $3$:
\[\gon(\tetrahedron)=3.\]

A natural question to ask is whether there are any other strategies to win the Gonality Game on a tetrahedron with \(3\) chips.  We can pick any vertex \(v\) and place \(3\) chips on it (so that's four possibilities), or place \(1\) chip on every vertex besides \(v\) (that's four more possibilities, although they're the same as the other four up to chip-firing).  See if you can prove that these are the only solutions, not just for the tetrahedron but for $K_n$!  That is, prove that with \(n-1\) chips on \(K_n\), the only ways to win the Gonality Game are to put all the chips on one vertex, or all the chips on different vertices.  (Hint: try to adapt the Dhar's burning algorithm argument we used.)

\begin{exploringfurther}[\textbf{Chip-firing games and parking functions}]
    It turns out there's lots more to say about chip-firing games on complete graphs.  Let's play a slightly different chip-firing game.  Say we have a complete graph with vertices $v_1,\ldots,v_n$, and we place $-1$ chips on $v_n$.  Now let's try to place (nonnegative numbers of) chips on $v_1,\ldots,v_{n-1}$ so that the Dollar Game \emph{isn't} winnable.  For instance, if $n=4$, any of the placements in Figure \ref{figure:unwinnable_divisors} will work, and it turns out that's all of them (we include the $0$ labels here for clarity).

\begin{figure}[hbt]
    \centering
    \includegraphics{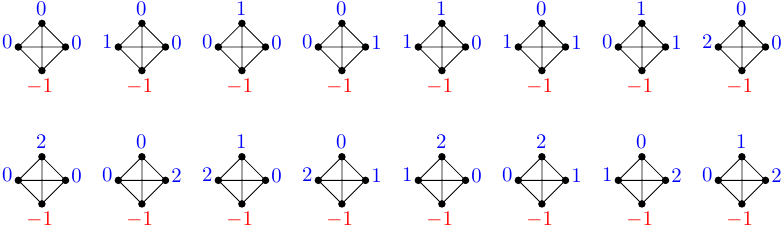}
    \caption{Unwinnable placements of chips on $K_4$}
    \label{figure:unwinnable_divisors}
\end{figure}

Let's represent those divisors as triples of integers, recording how many chips are on $v_1$, $v_2$, and $v_3$ (since $v_4$ always gets $-1$ chip, we'll ignore it).  We then have $16$ triples:
\[(0,0,0),\,(1,0,0),\,(0,1,0),\,(0,0,1),\,(1,1,0),\,(1,0,1),\,(0,1,1),\,(2,0,0),\,\]
\[(0,2,0),\,(0,0,2),\,(2,1,0),\,(2,0,1),\,(1,2,0),\,(0,2,1),\,(1,0,2),\,(0,1,2).\,\]
Add the number $1$ to each coordinate, giving $16$ more triples:
\[(1,1,1),\,(2,1,1),\,(1,2,1),\,(1,1,2),\,(2,2,1),\,(2,1,2),\,(1,2,2),\,(3,1,1),\,\]
\[(1,3,1),\,(1,1,3),\,(3,2,1),\,(3,1,2),\,(2,3,1),\,(1,3,2),\,(2,1,3),\,(1,2,3).\,\]
These triples are known as \emph{parking functions} of length $3$, and arise in many combinatorial contexts \cite{Handbook_of_enumerative_combinatorics}.  Remarkably, performing this process on $K_n$ always yields precisely the parking functions of length $n-1$.
\end{exploringfurther}

\section{The Octahedron}

Our next Platonic solid is the octahedron.  For the corresponding graph $\octahedron$, we're going to introduce two theorems that help us in studying gonality, one giving an upper bound and the other giving a lower bound.

For the upper bound, we're going to think about collections of vertices in the graph such that no two of them are connected by an edge; we call such a collection an \emph{independent set}.  It's always possible to find an independent set on a graph--for instance, you could take a set consisting of a single vertex.   Usually we can find larger sets, as long as our graph is not complete:  just pick two vertices that aren't joined by an edge.  Given a graph $G$, a great question to ask is:  how large of an independent set of vertices can we find in $V(G)$?  That maximum number, written \(\alpha(G)\), is called the \emph{independence number} of $G$.  Two independent sets on a graph are illustrated in Figure \ref{figure:independent_sets}, one of size \(2\) and one of size \(3\).  It turns out that \(3\) is the largest possible size of an independent set (try to argue this!), so for that graph \(G\) we have \(\alpha(G)=3\).

\begin{figure}[hbt]
    \centering
\includegraphics{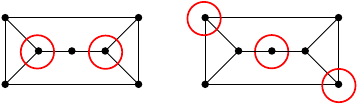}
    \caption{Two independent sets on the same graph.}
    \label{figure:independent_sets}
\end{figure}

The octahedron graph ends up having independence number \(2\).  We can certainly find an independent set of size \(2\) by picking two vertices that aren't connected by an edge; geometrically, this is like picking opposite corners of the solid.  To see we can't find any more than that, note that each vertex is connected to \emph{every other vertex but one}.  Thus if we had three vertices \(u,v,\) and \(w\), at least one of \(v\) and \(w\) would be connected to \(u\) by an edge.  Thus it is impossible to have an independent set  of size \(3\) or more.

Why do we care about independence number?  It turns out it lets us compute an upper bound on gonality!

\begin{theorem}[Proposition 3.1 in \cite{gonality_of_random_graphs}]\label{theorem:independence_number} Let $G$ be a simple\footnote{A \emph{simple} graph is one without multiple edges between any two vertices. For instance, all our Platonic graphs are simple.} graph with \(n\) vertices. We have \(\gon(G)\leq n-\alpha(G)\).
\end{theorem}

So for our octahedron, we have \(\gon(\octahedron)\leq 6-2=4\).  Why should this theorem hold for any simple graph?  Well, let's consider an independent set of maximum possible size.  Then, place a chip on every vertex that's \emph{not} in that independent set.  That placement has a total of \(n-\alpha(G)\) chips, and wins the Gonality Game:  Player B will have to place a chip on some vertex \(q\) in the independent set, so every neighbor of $q$ has a chip, and by firing \emph{every vertex besides $q$} we have a net movement of \(1\) chip from each neighbor of $q$ to $q$.  That eliminates the debt without introducing any new  debt, winning the Gonality Game.  See Figure \ref{figure:gonality_winning_octahedron} for a maximum independent set on the left, and the corresponding divisor winning the Gonality Game by placing a chip on every vertex outside the independent set.

\begin{figure}[hbt]
    \centering
    \includegraphics[scale=0.8]{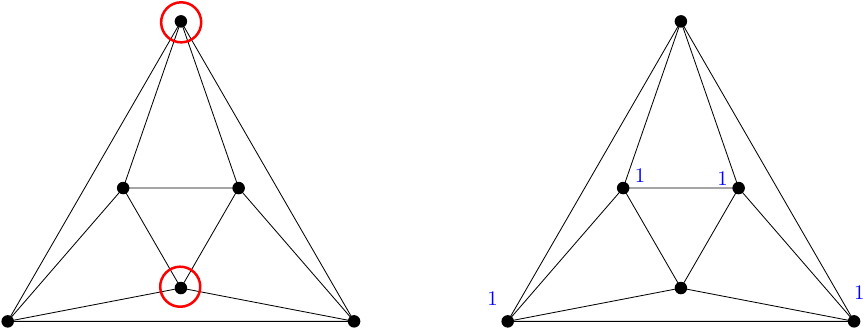}
    \caption{A maximum independent set on the octahedron, and the corresponding divisor that wins the Gonality Game}
    \label{figure:gonality_winning_octahedron}
\end{figure}

It's worth noting that \(\alpha(G)\) is not an easy number to compute for a general graph \(G\).  You might hope that you could build a maximum independent set by picking any vertex, then adding another vertex not incident to that one, and so on until no more vertices can be added.  But such a ``greedy'' algorithm might not always give you an independent set of size  \(\alpha(G)\).  This gets at the distinction between \emph{maximum}--meaning of the largest possible size--and \emph{maximal}--meaning nothing else can be added to it.  You can see an example of an independent set that is maximal, but not maximum, on the left in Figure \ref{figure:independent_sets} (with a maximum independent set on the right).

Now we turn to finding a lower bound on gonality, which will use another graph invariant: \emph{treewidth}.
Treewidth is a much-studied graph parameter that has been calculated for various families of graphs, and there exists software that can calculate the treewidth of sufficiently small graphs. Like gonality, treewidth is NP-Hard to compute; however unlike gonality, treewidth behaves nicely under different graph operations. Specifically, treewidth is \emph{minor monotone} (i.e., contracting or deleting edges as well as deleting vertices will always yield a graph with lower treewidth than the original). Before we can define treewidth, we must define brambles and how to calculate their bramble order.

Let $\mathcal{B}$ be a collection of subsets of vertices of $G$ such that each set of vertices forms a connected induced subgraph\footnote{The \emph{induced subgraph} of a collection $B$ of vertices is the graph with vertex set $B$ and edge set consisting of all edges in $G$ that happen to have both endpoints in $B$.}. We call $\mathcal{B}$ a \emph{bramble} if any two sets $S$ and $T$ in $\mathcal{B}$ ``touch'', meaning that either $S$ and $T$ have a vertex in common, or there exists an edge in $E(G)$ that connects a vertex in $S$ to a vertex in $T$.  See Figure \ref{figure:bramble} for an example on a $2\times 3$ grid graph, with the circled collections of vertices forming the sets that make up the bramble.

\begin{figure}[hbt]
    \centering
    \includegraphics{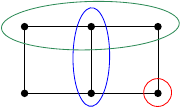}
    \caption{A bramble on a $2\times 3$ grid graph consisting of three sets of vertices}
    \label{figure:bramble}
\end{figure}

Given a bramble $\mathcal{B}$, we call a set of vertices $W$ a \emph{hitting set} of $\mathcal{B}$ if for every set of vertices in $\mathcal{B}$, each set contains at least one vertex from $W$. The smallest possible size of a hitting set is called the \emph{(bramble) order} of $\mathcal{B}$.  The bramble in Figure \ref{figure:bramble} has bramble order $2$, with the top middle vertex and the bottom right vertex forming a minimum hitting set.

Given a graph $G$, let $n$ be the largest possible bramble order on $G$. Then the \emph{treewidth} of $G$, written as $\tw(G)$, is $n-1$.  It's worth noting that this isn't the usual definition of treewidth; however, our definition is equivalent by work in \cite{st-graph-searching-treewidth}, and happens to be more useful for our purposes.  While treewidth has plenty of uses in graph theory, we care about treewidth because it is a lower bound on gonality.

\begin{theorem}[Theorem 3.1 in \cite{debruyn2014treewidth}]
For any graph $G$, $\tw(G) \leq \gon(G)$.
\end{theorem}

This actually gives an alternate way to lower bound the gonality of the tetrahedron!  For instance, for $K_4$, we can construct a bramble of order $4$ by just considering all $1$-element sets of vertices:  $\mathcal{B}=\{\{v_1\},\{v_2\},\{v_3\},\{v_4\}\}$. Every vertex touches every vertex, so this is indeed a bramble; and the smallest hitting set for this bramble consists of every single vertex (leave any out, and you'll miss one of the sets).  So the order of the bramble is $4$, meaning that tree-width of the graph is at least $3$, meaning the gonality is at least $3$. 

Let's now construct a bramble of order \(5\) on the octahedron.  Label the vertices \(u_1,u_2,v_1,v_2,w_1,w_2\), where a vertex is connected only to those vertices with a different letter label.  Then consider the bramble consisting of the following six collections of vertices:  \[\{u_1\}, \{v_1\}, \{w_1\}, \{u_2,v_2\}, \{u_2,w_2\},\{v_2,w_2\}.\]  Any hitting set will definitely need to have each of \(u_1\), \(v_1\), and \(w_1\) (that's true anytime you're trying to hit a set with only one element), as well as at least two of \(u_2,v_2,\) and \(w_2\) (if we only include one, the other two form an unhit set).  Thus any hitting set has at least five elements, and there does exist a hitting set with exactly five elements (for instance, every vertex except \(w_2\)), so the order of this bramble is \(5\).  That means that \(\tw(\octahedron)\geq 4\).  Combined with our upper bound of \(\gon(\octahedron)\leq 4\), we have
\[4\leq \tw(\octahedron)\leq \gon(\octahedron)\leq 4.\]
Since that string of inequalities starts and ends with the same number, we know that in fact \(\tw(\octahedron)=\gon(\octahedron)=4\).  So, we can use treewidth to successfully complete our computation of the gonality of the octahedron!

Like independence number, treewidth is hard to compute.  Luckily, there is an easy-to-compute number that relate nicely to treewidth giving us an easier (though in general weaker) lower bound on gonality.
 For any vertex \(v\) in \(G\), we call the number of edges incident to \(v\) the \emph{degree} of \(v\); and we let \(\delta(G)\) denote the minimum degree of a vertex of \(G\).  It turns out that for simple graphs \(G\), we have that \(\delta(G)\) is a lower bound on treewidth, giving us the string of inequalities
\[\delta(G)\leq \tw(G)\leq \gon(G).\]
So an even faster argument that the octahedron has gonality at least \(4\) is note that it's minimum degree is \(4\), providing us the desired lower bound!  There do exist graphs where treewidth is strictly larger than minimum degree, so in general we're better off using treewidth, but \(\delta(G)\) can be computed so quickly that it's often worth checking before delving into a lengthy search for brambles of high order.

\begin{exploringfurther}[\textbf{The gonality of complete multipartite graphs}]
Before we move on to our next Platonic solid, it's worth mentioning that just as the tetrahedron was one instance of complete graph, the octahedron is an instance of an oft-studied infinite family of graphs:  the \emph{complete multipartite graphs}.  Given integers \(n_1\leq n_2\leq\cdots\leq n_k\), the complete multipartite graph \(K_{n_1,n_2,\ldots,n_k}\) is the graph on \(n=n_1+n_2+\cdots +n_k\) vertices with \(n_1\) vertices in one cluster, \(n_2\) vertices in another cluster, and so on, such that two vertices are connected if and only if they're not in the same cluster. A few multipartite graphs are illustrated in Figure \ref{figure:complete_multipartite_graphs}.

\begin{figure}[hbt]
    \centering
    \includegraphics{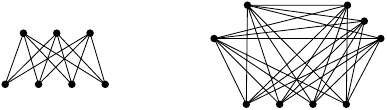}
    \caption{The complete multipartite graphs $K_{3,4}$ and $K_{2,3,4}$}
    \label{figure:complete_multipartite_graphs}
\end{figure}

Remember how the octahedron has six vertices, and every vertex is connected to every other vertex except one?  That means that the octahedron is precisely the complete multipartite graph \(K_{2,2,2}\). It's natural to ask then if our calculation of the gonality of the octahedron is a special case of a more general formula--and it turns out it is!  In fact, we have
\[\gon(K_{n_1,n_2,\ldots,n_k})=n-n_k.\]
So the gonalities of the graphs in Figure \ref{figure:complete_multipartite_graphs} are $3$ and $5$, respectively.  This formula was first proved in \cite{debruyn2014treewidth}, but before you look at their argument, see if you can prove it for yourself!  Try to generalize our independence number argument for an upper bound, and our treewidth argument for a lower bound.  (Or, stay tuned for another lower bound technique that we'll learn very shortly.)
\end{exploringfurther}

\section{The Cube}

Our next Platonic solid is perhaps the most famous:  the cube.  Although we can flatten the cube into a graph as we've done for our previous examples, we can also very nicely write down three-dimensional coordinates for its vertices:  they're all possible combinations of \(0\)'s and \(1\)'s.  The edges in turn have a very nice description as well--they connect two vertices precisely when their coordinates differ in exactly one spot!  This labelling scheme will be helpful for thinking about this graph.

Let's start by upper bounding the gonality of the cube.  One strategy we've already seen is to find an independent set of maximum possible size.  We can find an independent set of size \(4\) by picking all vertices whose coordinates sum to an even number--namely \((0,0,0)\), \((1,1,0)\), \((1,0,1)\), and  \((0,1,1)\).  It turns out that's the best we can do; try to argue that there aren't any independent sets with \(5\) or more vertices.  (Hint:  think about splitting the vertices of the cube into two parallel squares, and do a counting argument.)  That means we have \(\gon\left(\cube\right)\leq 8-4=4\).  

Before we move on to a lower bound, it's worth mentioning a different strategy that also shows the gonality of the cube is at most \(4\). It turns out that a cube is an example of a \emph{Cartesian product of graphs}.  Given two graphs \(G\) and \(H\), say where \(G\) has \(m\) vertices and \(H\) has \(n\) vertices, we can build a new graph called their Cartesian product \(G\square H\) with \(m\cdot n\) vertices arranged in an \(m\times n\) grid, such that the columns look like copies of \(G\) and the rows look like copies of \(H\).  You can see an example in Figure \ref{figure:product_of_graphs}, along with a few chip placements on \(G\square H\).

\begin{figure}[hbt]
    \centering
\includegraphics{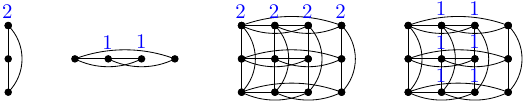}
    \caption{A graph $G$ on $3$ vertices, a graph $H$ on $4$ vertices, and two copies of the Cartesian product $G\square H$ on $12$ vertices.  Each divisor illustrated wins the Gonality Game on its graph.}
    \label{figure:product_of_graphs}
\end{figure}
Anytime a graph is the product of two other graphs, we have a certain upper bound on its gonality.

\begin{theorem}[Proposition 3 in \cite{aidun2019gonality}]
    We have \(\gon(G\square H)\leq |V(G)|\gon(H)\), as well as \(\gon(G\square H)\leq |V(H)|\gon(G)\).
\end{theorem}

Here's the idea behind that theorem: pick a chip placement on \(H\) that wins the Gonality Game with as few chips as possible, i.e. with \(\gon(H)\) chips.  Place that divisor on every copy of \(H\) in \(G\square H\), so that we have a total of \(|V(G)|\gon(H)\) chips.  It turns out that divisor wins the Gonality Game on \(G\square H\): wherever Player B places debt, focus on that copy of \(H\), and think about what chip-firing moves you would do on \(H\) to eliminate the debt.  Then, perform those moves on \emph{all} copies of \(H\)!  This will eliminate the debt, without introducing any more.  A symmetric argument works if we swap \(G\) and \(H\).  These two chip placements are illustrated for our example graph in Figure \ref{figure:product_of_graphs}, giving us upper bounds of $8$ and $6$ on its gonality; of course, we'll always take the smaller of the two, so we know that graph has gonality at most $6$.

One way of building the cube graph is as follows.  Start with \(K_2\), the complete graph on two vertices.  Take the product \(K_2 \square K_2\).  This is a graph on four vertices that looks like a square or a cycle, often written as \(C_4\).  Then, take another product with \(K_2\), giving us \(C_4\square K_2\).  That, it turns out, is the cube graph.  Using our theorem on product graphs, this gives us \(|V(C_4)|\gon(K_2)=4\cdot 1=4\) and \(|V(K_2)|\gon(C_4)=2\cdot 2=4\) as upper bounds on the gonality of the cube, so we have an alternate argument that \(\gon(\cube)\leq 4\).  See Figure \ref{figure:two_gonality_divisors_cube} for two placements of four chips each that win the Gonality Game on the cube; the left comes from the independent set strategy, and the right comes from the Cartesian product strategy.

\begin{figure}[hbt]
    \centering
    \includegraphics[scale=0.6]{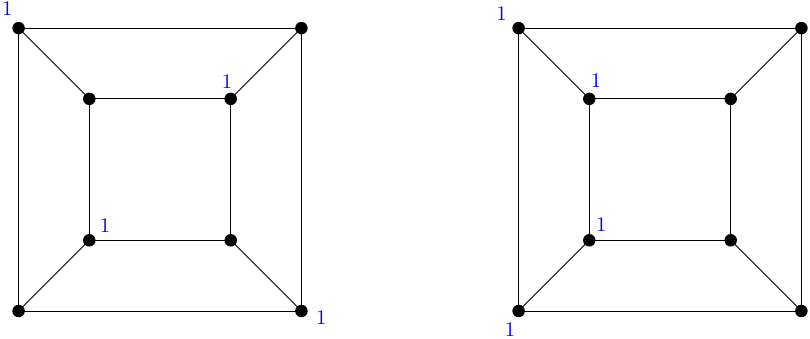}
    \caption{Two divisors that win the Gonality Game on the cube. The first places a chip on each vertex besides those in an independent set; and the second places a chip on each vertex in a copy of $C_4$, viewing the cube as $C_4\square K_2$}
    \label{figure:two_gonality_divisors_cube}
\end{figure}

Now let's turn to the lower bound.  It turns out that the treewidth of the cube is \(3\), so that's the best lower bound we can get on gonality using our strategy for the octahedron.  We could perhaps try a Dhar's-based argument (a Dhargument, if you will) like we did for the tetrahedron to argue that \(3\) chips cannot win the Gonality Game, but that ends up taking a while.  Instead, we're going to introduce a new invariant, called the \emph{scramble number} of a graph.

A \emph{scramble} on a graph is any collection of sets of vertices whose induced subgraphs are connected; in other words, it's a bramble where we've dropped the requirement that every pair of vertex sets ``touch'' one another.  So every bramble is a scramble, but not all scrambles are brambles.  Every set of vertices in a scramble is called an \emph{egg} (since you can't make a scramble without eggs!). See Figure \ref{figure:scrambles} for a few examples of scrambles--note that the eggs might overlap, or might be disjoint, and that the eggs might cover all vertices in the graph or they might not.

\begin{figure}[hbt]
    \centering
    \includegraphics{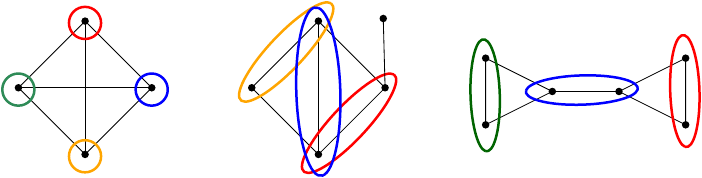}
    \caption{Three graphs, each with a scramble.  The first two scrambles are also brambles, but the third is not.}
    \label{figure:scrambles}
\end{figure}

Given a scramble \(\mathcal{S}\), there are two numbers we need to compute.  The first is the hitting number, \(h(\mathcal{S})\), which is familiar from our work with brambles:  it's the smallest number of vertices we can choose to so that we have at least one vertex in each egg in \(\mathcal{S}\).  The second number asks how easy it is to separate some pair of eggs in the scramble.  More formally, an \emph{egg-cut} for a scramble is a collection of edges that, if deleted, would split the graph into two pieces, each containing at least one complete egg.  The the egg-cut number of \(\mathcal{S}\), written \(e(\mathcal{S})\), is then the smallest size of an egg-cut\footnote{If every pair of eggs overlaps, then there don't exist any egg-cuts, in which case we take \(e(\mathcal{S})\) to be \(\infty\).}.  For instance, in the graphs from Figure \ref{figure:scrambles}, we have hitting numbers of $4$, $2$, and $3$, and egg-cut numbers of $3$, $3$, and $1$.  (Note that the egg-cut number might not be achieved by deleting all edges coming directly out of an egg; for instance, for the third scramble, a minimum egg-cut is achieved by deleting the middle edge to separate the green and red eggs. Thus to compute $e(\mathcal{S})$, it's not enough to look at each egg and see how many edges connect it to the rest of the graph.)

Once we have these two numbers for a scramble \(\mathcal{S}\), we can compute the \emph{(scramble) order}\footnote{The reason we sometimes need to specify ``bramble order'' and ``scramble order'' is that a bramble can be treated as a scramble, and so has two possible orders.} of \(\mathcal{S}\), which is defined to be the minimum of the hitting number and the egg-cut number:
\[||\mathcal{S}||=\min\{h(\mathcal{S}),e(\mathcal{S})\}.\]
So from left to right, the scrambles in Figure \ref{figure:scrambles} have orders $3$, $2$, and $1$.  Now we can finally define the scramble number of a graph \(G\), denoted \(\sn(G)\):  it's the maximum order of any scramble on \(G\).

Why do we like scramble number?  It turns out that it's an even better lower bound on gonality than treewidth!
\begin{theorem}[Theorem 1.1 in \cite{new_lower_bound}] For any graph \(G\), we have \(\tw(G)\leq \sn(G)\leq \gon(G)\).
\end{theorem}

To see the power of scramble number in action, let's consider a scramble \(\mathcal{S}\) on our cube, illustrated in Figure \ref{figure:cube_scramble}.  Since no two eggs overlap, we have to pick a vertex from each to hit them all, so we immediately find \(h(\mathcal{S})=4\).  There definitely exist egg-cuts with four edges--for instance, delete all the edges coming out of an egg, or delete the four horizontal edges.  How can we be sure there isn't an egg-cut with three edges or fewer? 

\begin{figure}[hbt]
    \centering
    \includegraphics[scale=0.8]{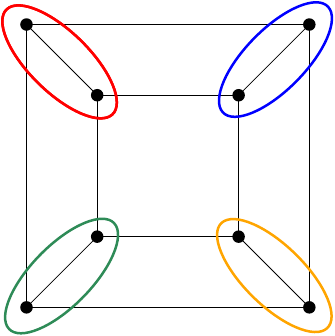}
    \caption{A scramble on the cube}
    \label{figure:cube_scramble}
\end{figure}

Well, pick any pair of our four eggs.  Let's consider four different ways to travel from the first to the second: 
 start on either the inner or outer vertex of the first egg, and travel either clockwise or counterclockwise until we reach a vertex of the second egg.  That gives a total of four different paths connecting one egg to the other, and these paths don't have any edges in common with one another.  In any egg-cut separating those two eggs, each path has to be broken; otherwise, the two eggs would be in the same component of the graph after we delete the edges.  That means the egg-cut must include at least one edge from each of the four paths; in other words, the egg-cut must have at least four edges in it!  That's true for any pair of eggs, so we have \(e(\mathcal{S})\geq 4\), and thus \(e(\mathcal{S})=4\).

 Where does that leave us?  Well, the order of this scramble is the minimum of the hitting number and the egg-cut number, which is \(4\):
 \[||\mathcal{S}||=\min\{h(\mathcal{S}),e(\mathcal{S})\}=\min\{4,4\}=4.\]
 The scramble number of the cube is the maximum order of any scramble, so we have that it's at least as large as the order of any given scramble:
 \[4=||\mathcal{S}||\leq \sn(G).\]
 Now we use the facts that \(\sn(G)\leq \gon(G)\) and \(\gon(G)\leq 4\) to get a string of inequalities:
 \[4\leq \sn(G)\leq \gon(G)\leq 4.\]
 Much like in our octahedron argument, this string of inequalities forces every number in between to be \(4\), so we have \(\sn(G)=\gon(G)=4\).  So, scramble number lets us compute the gonality of the cube!

\begin{exploringfurther}[\textbf{The gonality of the $n$-dimensional hypercube}]

 Before we leave our six-sided solid, it's worth asking how far we can push this.  The cube falls very naturally into a broader family of graphs, called the \emph{$d$-dimensional hypercubes} \(Q_d\).  One way to define \(Q_d\) is as having \(2^d\) vertices, one for each string of \(0\)'s and \(1\)'s of length \(d\), with two vertices connected by an edge precisely when their strings differ in exactly one digit.  So \(Q_1\) has two vertices joined by an edge; \(Q_2\) is a square; and \(Q_3=\cube\) is our cube.  Alternatively, once we have \(Q_d\), we can build \(Q_{d+1}\) by making two copies of \(Q_d\), and attaching the matching pairs of vertices (like how we can build a cube by drawing two copies of a square, and connecting matching vertices); graph theoretically, this is defining $Q_{d+1}=Q_{d}\square K_2$.

 The graph \(Q_1\) is a \emph{tree}, meaning it is connected but contains no cycles\footnote{A cycle is a sequence of two or more vertices in the graph, with only the first and last equal to each other, such that each consecutive pair of vertices is connected by an edge.}. It turns out that every tree has gonality equal to \(1\); see if you can argue that any single chip placed anyywhere on a tree will win the Gonality Game.  The graph \(Q_2\) is an example of a \emph{cycle}, meaning it consists of \(n\) vertices arranged in one big cycle.  All cycle graphs turn out to have gonality \(2\); see if you can prove this, using the tools we've build up so far!  That means that the gonalities of \(Q_1\), \(Q_2\), and \(Q_3\) are \(1\), \(2\), and \(4\).

 One possible pattern we might spot is that for \(d\leq 3\), we have \(\gon(Q_d)=2^{d-1}\); indeed, such a formula is conjectured to hold for all $d$ in \cite{debruyn2014treewidth}.  That number will always be an upper bound on \(\gon(Q_d)\); either argument we made for the upper bound on the cube (independence number or Cartesian product) gets us that.  So, when do we have equality?  In the same paper that introduced scramble number, the authors proved that \(\gon(Q_4)=8\), so the pattern continues\footnote{They actually studied \(Q_4\) in a slightly different guise: as a \emph{$4\times 4$ toroidal grid graph}.}. And in \cite{uniform_scrambles}, the authors proved that \(\gon(Q_5)=16\).  Both these arguments used scramble number, so we might hope we can just keep building scrambles on these graphs to prove that our pattern of \(2^{d-1}\) keeps on holding!

 Sadly, \cite{uniform_scrambles} also proved that for \(d\geq 6\), we have 
\(\sn(Q_d)<2^{d-1}\).  So there's still hope that \(\gon(Q_d)\) might be equal to \(2^{d-1}\), but we definitely won't be able to prove it using scramble number.  We'll probably need to introduce some new tools if we ever want to prove a result like that!
\end{exploringfurther}

\section{The Dodecahedron}

We've got two Platonic solids left:  the dodecahedron and the icosahedron.  Both of these will take more effort than the earlier solids, especially when it comes to lower bounding gonality.  Fortunately for the dodecahedron, we've already seen the main tool we'll need, namely scramble number; it'll just take a lot of work to implement it!

\begin{figure}[hbt]
    \centering
    \includegraphics[scale=0.7]{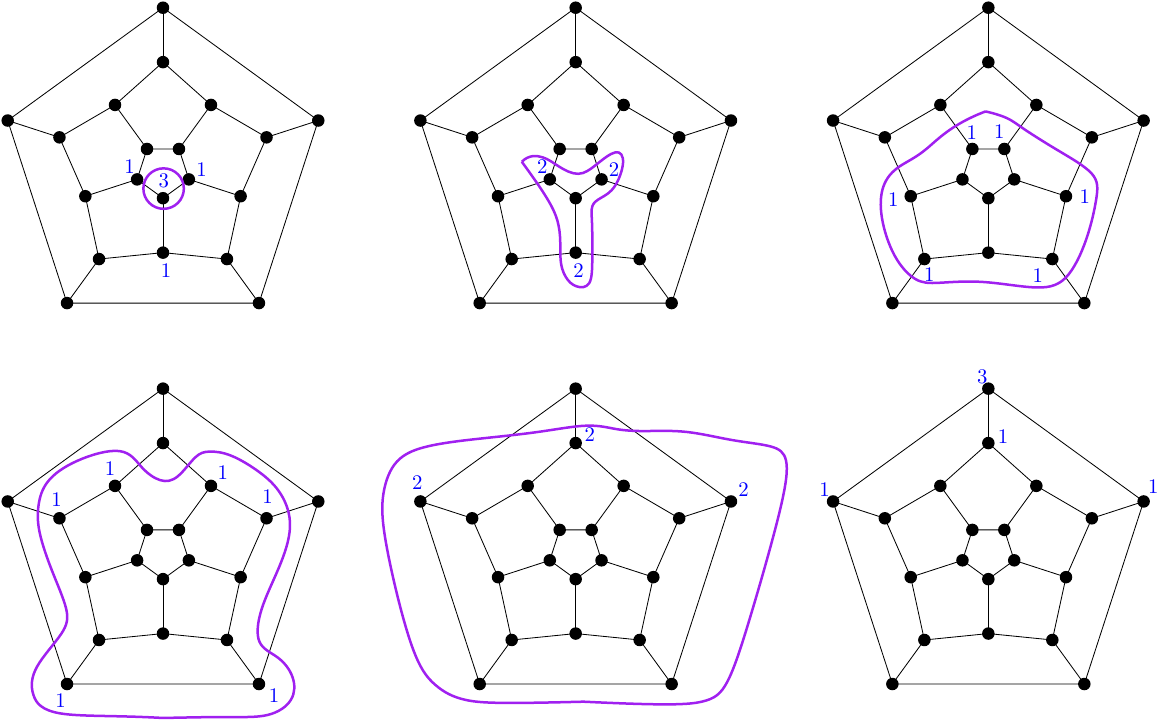}
    \caption{A chip placement that wins the Gonality Game on the dodecahedron.  To get from one divisor to the next, set-fire the circled collection of vertices.}
    \label{figure:gonality_winning_dodecahedron}
\end{figure}

Before we get into lower bounding for the dodecahedron's gonality, let's find an upper bound.  Place three chips on any vertex you like, and one chip of each of its neighbors; by the symmetry of the dodecahedron, it actually doesn't matter which vertex you start from!  Fire the first vertex, so that its three neighbors each have two chips.  Then fire the starting vertex together with its neighbors, spreading out the chips across six vertices.  From here, fire all the vertices fired so far--that will move the chips to six new distinct vertices.  Again firing all vertices we've fired, plus the newly chipped ones, gives us two chips on three vertices, all of which have a common neighbor; firing every vertex but that one puts chips onto it.  This process is illustrated in Figure \ref{figure:gonality_winning_dodecahedron}.  Since that placement of six chips could send a chip anywhere in the graph without introducing debt, no matter where Player B puts their $-1$ debt it can be eliminated. Thus this chip placement wins the Gonality Game, and \(\gon(\dodecahedron)\leq 6\).  



In order to show that \(\gon(\dodecahedron)\geq 6\), we will find a scramble \(\mathcal{S}\) on the graph and show that it has order at least \(6\). That gives a lower bound on \(\sn(\dodecahedron)\), which in turns gives  a lower bound on \(\gon(G)\), and since the upper and lower bounds match we'll be able to conclude \(\gon(\dodecahedron)= 6\).  (You might wonder whether treewidth would work here instead of scramble number--computationally at least, the answer is yes!  Use SageMath \cite{sagemath}, one can compute that $\tw(\dodecahedron)=6$, which implies $\gon(\dodecahedron)=6$.  However, we're shooting for a proof we can write down, and we don't know a nice description for a bramble of bramble order $7$ for us to work with.)

The scramble that ends up working is a pretty big one:  it's the \emph{$6$-uniform scramble on \(\dodecahedron\)}.  That's the scramble \(\mathcal{S}\) containing every possible egg on six vertices.  In other words, a collection of six vertices is an egg in \(\mathcal{S}\) if and only if they induce a connected subgraph of \(\dodecahedron\).  There are a lot of eggs in that scramble; for instance, there are \(12\) pentagons (cycles of length \(5\)) in \(\dodecahedron\), and each borders \(5\) other vertices, so there are \(12\cdot 5=60\) eggs of the form ``vertices of a pentagon plus one more vertex that neighbors the pentagon''.  There are plenty of other eggs too, which end up being trees.

For such a big scramble, it'll take a lot of work to compute its order.  We'll be a little more formal with our proofs here, to make sure we don't let anything slip through the cracks.  Let's start with the hitting number.  In the proof below, we'll focus on some different parts of the dodecahedron graph based on how we've drawn it:  the outer pentagon (the outermost five vertices), the inner pentagon (the innermost five vertices), and the middle $10$-cycle (the remaining $10$ vertices, which induce a cycle of length $10$).

\begin{lemma} The hitting number of the \(6\)-uniform scramble \(\mathcal{S}\) on \(\dodecahedron\) is at least \(6\):
\[h(\mathcal{S})\geq 6.\]
\end{lemma}

\begin{proof}
     Suppose for the sake of contradiction we do not have $h(\mathcal{S})\geq 6$, so \(h(\mathcal{S})\leq 5\).  Let \(W\) be a hitting set for $\mathcal{S}$ consisting of only \(5\) vertices.

    First, we'll argue that \(W\) has to hit every pentagon.  Suppose not!  Every pentagon looks the same, so we can go ahead and assume that the outer pentagon is missed.  Since $W$ hits every connected subgraph on $6$ vertices, it must hit every one of the five inner vertices bordering the outer pentagon; otherwise one of those vertices together with the pentagon would be an unhit egg.  That accounts for all five vertices in $W$, as illustrated on the left in Figure \ref{figure:hitting_set_dodecahedron}.  But, that means all vertices in the central pentagon are unhit, as are the five vertices bordering it, giving us at least $5$ eggs that are unhit, a contradiction.  So, $W$ must hit every pentagon.  
    
    \begin{figure}[hbt]
        \centering
        \includegraphics{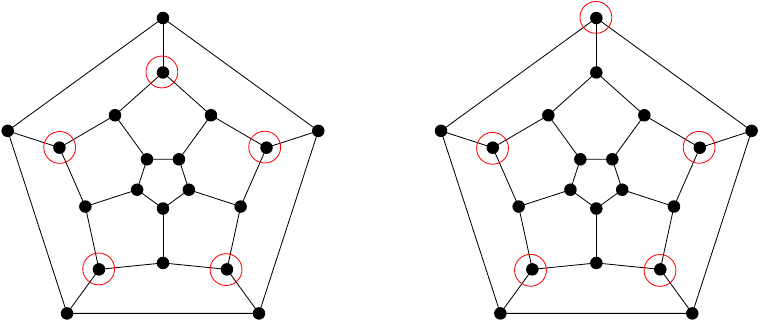}
        \caption{Two possible structures of a potential hitting set with $5$ vertices.  The first occurs is a the outer pentagon is unhit; the second occurs if the outer pentagon is hit exactly once.}
        \label{figure:hitting_set_dodecahedron}
    \end{figure}
    

    Now we consider the three pentagons illustrated in Figure \ref{figure:three_pentagons}, which have no overlapping vertices.  Since every pentagon is hit by $W$ and there are only five vertices in $W$, one of those three pentagons is hit exactly once by $W$; without loss of generality we assume it is the outer pentagon that is hit once, and that it is the topmost vertex that is hit.

    \begin{figure}[hbt]
        \centering
    \includegraphics[scale=0.8]{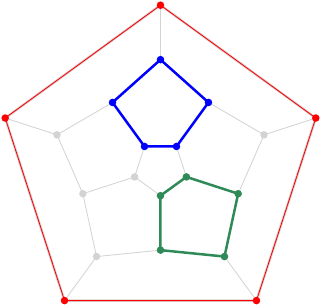}
        \caption{Three pentagons that share no vertices}
        \label{figure:three_pentagons}
    \end{figure}

    Consider any pair of adjacent vertices on the middle $10$-cycle not including the neighbor of the topmost vertex of the outer pentagon.  These two vertices, together with the four unhit vertices on the outer pentagon, form an egg, which must be hit on one of the $10$-cycle vertices.  This is equivalent to taking a path on $9$ vertices, and for each vertex either hitting it or hitting its neighbor.  This can be accomplished most efficiently by choosing $4$ vertices, alternating and away from the endpoint.  This completely determines $W$:  it consists of the topmost vertex, and then the four vertices on the middle $10$-cycle incident to another vertex on the outer pentagon, as illustrated on the right in Figure \ref{figure:hitting_set_dodecahedron}.  However, this still leaves the inner pentagon together with five adjacent vertices unit, so $W$ is not a hitting set.

    Having reached a contradiction, we conclude that $h(\mathcal{S})\geq 6$.
\end{proof}

All we have to show now is that \(e(\mathcal{S})\geq 6\).  Because of how we built our scramble, this is equivalent to the following claim:  if we delete edges and split the dodecahedron into two pieces, each connecting at least \(6\) vertices, then we must have deleted at least \(6\) edges.  There are only a few ways we could split the dodecahedron:  into \(6\) and \(14\) vertices, into \(7\) and \(13\), into \(8\) and \(12\), into \(9\) and \(11\), and into \(10\) and \(10\). Focusing on the smaller side, it makes sense to ask the question:  how many edges cpi;d contribute to the outdegree of a subgraph of \(\dodecahedron\) that has between \(6\)
and \(10\) vertices?  The number of edges of a subgraph \(H\) connecting it to the result of the graph is called the \emph{outdegree} of \(H\), written \(\outdeg(H)\).  So in order to prove \(e(\mathcal{S})\geq 6\), all we have to do is prove the following.

\begin{lemma}\label{lemma:subgraph_of_dodecahedron_outdegree}
    If \(H\) is a subgraph of \(\dodecahedron\) and \(6\leq |V(H)|\leq 10\), then \(\outdeg(H)\geq 6\).
\end{lemma}

Before proving this Lemma, we must recall one of the most famous results of graph theory.  A graph is called \emph{planar} if it can be drawn in a two-dimensional plane without any of its edges crossing.  For instance, all of our Platonic graphs are planar, as are any subgraphs.  It turns out for planar graphs that are connected, there is an universal relationship between the number of vertices, of edges, and of bounded faces in any planar drawing.

\begin{theorem}[Euler's Formula] Let $G$ be a connected planar graph, and assume we have a planar drawing of $G$ with $g$ bounded faces.  Then
\[g=|E(G)|-|V(G)|+1\]
\end{theorem}
Normally this is phrased in terms of \emph{all} faces, including the unbounded on (leading to the famous formula $V-E+F=2$, where $V$ counts vertices, $E$ counts edges, and $F$ counts faces).  However, the number of bounded faces comes up in a foundational result on chip-firing games on graphs, which we'll explore at the end of this section.  In the meantime, let's think about how many bounded faces $H$ could have in our lemma.  Note that any bounded face in $H$ must be a bounded face in $G$, or made up of combined such faces.  All the bounded faces are pentagons, which share at most two vertices, with at most three bounded faces meeting at a vertex.  This means in order to have two bounded faces, $H$ needs at least $8$ vertices:  $5+5=10$ to form two pentagons, possibly minus $2$ if the pentagons share two vertices.  And in order to have three bounded faces, $H$ needs at least $10$ vertices by a similar argument.

We are now ready to prove the lemma.

\begin{proof}[Proof of Lemma \ref{lemma:subgraph_of_dodecahedron_outdegree}]  Throughout we can assume that $H$ is an induced subgraph (i.e., it contains all possible edges from $G$), as this does not change $\outdeg(H)$. 
Note that each vertex in $G$ is incident to $3$ vertices. So to count the number of edges leaving $H$, we can multiply $3$ by the number of vertices in $H$, and then subtract off any edges that stayed interior to $H$.  In fact, we have to subtract these edges off twice, since they were counted by both their endpoints in $H$. Thus the outdegree of any induced subgraph $H$ is
\[3|V(H)|-2|E(H)|.\]

Now let's consider $g=g(H)$, the number of bounded faces in $H$. As previously argued, $g$ is at most $1$ if $|V(H)|\leq 7$, at most $2$ if $|V(H)|\leq 9$, and at most $3$ if $|V(H)|=10$.  Since $g=|E(H)|-|V(H)|+1$, we have $|E(H)|=g+|V(H)|-1$, so
\[3|V(H)|-2|E(H)|=3|V(H)|-2g-2|V(H)|+2=|V(H)|-2g+2.\]
Let's lower bound this expression in four cases:
\begin{itemize}
    \item If $g=0$, this is $|V(H)|+2$.  This is at least $8$.
    \item If $g=1$, this is $|V(H)|$. This is at least $6$.
    \item If $g=2$, this is $|V(H)|-2$.  Since $|V(H)|\geq 8$ in this case, this is at least $6$.
    \item If $g=3$, this is $|V(H)|-4$.  Since $|V(H)|=10$, this is $6$.
\end{itemize}
In every case, we find that the outdegree of $H$ is at least $6$, as desired.
\end{proof}

We now know that our $6$-uniform scramble has an egg-cut number of (at least) $6$.  Combined with the hitting number result, this means that $\sn(\dodecahedron)\geq ||\mathcal{S}||\geq 6$.  This gives us the desired lower bound on $\gon(\dodecahedron)$, letting us conclude that
\[\gon(\dodecahedron)=6.\]

\begin{exploringfurther}[\textbf{The Riemann-Roch Theorem for Graphs}]

Before we venture on to our last Platonic solid, we return to the formula $g=|E(G)|-|V(G)|+1$.  Although this formula came from planar graphs, we can define the number $g(G)=|E(G)|-|V(G)|+1$ for any graph $G$.  This number has many interesting properties; for instance, it is the largest number of edges you can delete from a graph while still keeping it connected.  If you view your graph as a topological space, or perhaps thickening up your graph and viewing it as a topological surface, this number $g(G)$ represents the number of ``holes'' in that space.  It is perhaps not too surprising from this topological context that some authors \cite{bn} have taken to calling this number the \emph{genus} of a graph; however, most graph theory texts reserve the term ``genus'' for another graph invariant, so \emph{cyclomatic number} or \emph{first Betti number} are less controversial terms for $g(G)$.

The number $g(G)$ plays a role in a beautiful and powerful result used in studying chip-firing games on graphs, which we briefly present here.  First, the \emph{canonical divisor} $K$ on a graph $G$ is the chip placement that puts $\val(v)-2$ chips on each vertex $v$.  You can show, with a little work, that $\deg(K)=2g(G)-2$.   Next, the \emph{rank} of a divisor $D$, written $r(D)$, measures how much added debt the divisor could eliminate on the graph through chip-firing moves, regardless of where that debt is placed.  This means that $r(D)\geq 0$ if and only if the Dollar Game is winnable on $D$, and $r(D)\geq 1$ if and only if $D$ wins the Gonality Game.  More generally, $r(D)\geq k$ if and only if $D$ wins a generalization of the Gonality Game where Player B gets to place $-k$ chips, distributed however they'd like throughout the graph.  (By convention, if the Dollar Game is not winnable on $D$, we set $r(D)=-1$.)

Perhaps unsurprisingly, the rank of a divisor is very hard to compute.  However, there is a very easy-to-compute relationship between the rank of $D$ and the rank of $K-D$.

\begin{theorem}[The Riemann-Roch Theorem for Graphs, \cite{bn}]
    For any divisor $D$ on a graph $G$, we have
    \[r(D)-r(K-D)=\deg(D)+1-g(G).\]
\end{theorem}
This strongly parallels a classical result from Algebraic Geometry called the Riemann-Roch Theorem for Algebraic Curves.  It is remarkable that such a similar result holds for chip-firing games on graphs!

For a quick glimpse of how powerful this result is, suppose for a moment that $D$ is a divisor (possibly with debt) that has a total of $g(G)$ chips.  Then we have
\[r(D)=\deg(D)+1-g(G)+r(K-D)=g(G)+1-g(G)+r(K-D)=1+r(K-D).\]
No matter what divisor $K-D$ is, its rank must be at least $-1$, so $r(D)=1+r(K-D)\geq 0$.  So $D$ has nonnegative rank, meaning that the Dollar Game starting from $D$ must be winnable.  We get to conclude this just based on the degree of $D$, without knowing any further information!  With a little work, you can adapt this argument to show that if $D$ has at a total of $g+1$ chips, then $r(D)\geq 1$, meaning that $D$ wins the Gonality Game.  That means we have a completely universal upper bound on gonality:
\[\gon(G)\leq g(G)+1.\]
There is a conjectural upper bound on $\gon(G)$ that is stronger than this by around a factor of two:
\[\gon(G)\leq \frac{g(G)+3}{2};\]
however, whether this bound always holds is still an open question \cite{baker}.
\end{exploringfurther}

\section{Icosahedron}
Our final Platonic solid is the icosahedron.  This will turn out to be the only Platonic solid \emph{where scramble number is not powerful enough to compute gonality!}  But let's not get ahead of ourselves--let's see what upper and lower bounds we can come up with.

For an upper bound, we can compute the independence number of \(\icosahedron\) and use Theorem \ref{theorem:independence_number}.  It turns out that there do exist independent sets of size \(3\) on the icosahedron, and that no larger independent sets are possible (try to prove this!), so we have \(\gon(\icosahedron)\leq 12-3=9\).  You can check out an independent set of size \(3\), and a placement of $9$ chips that wins the Gonality Game, in Figure \ref{figure:icosahedron_independent_set_divisor}.

\begin{figure}[hbt]
    \centering
    \includegraphics[scale=0.8]{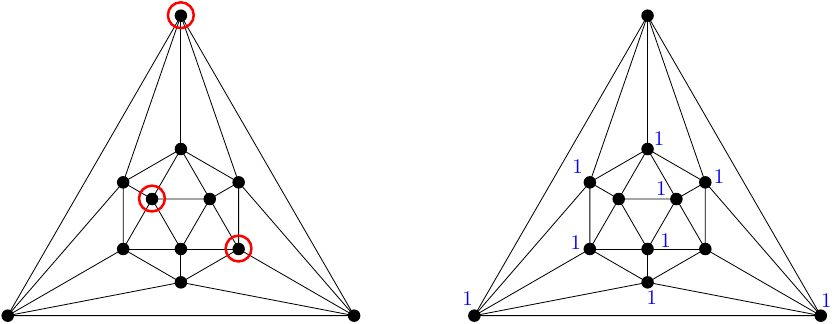}
    \caption{An independent set of size $3$ on the icosahedron, and a divisor of degree $9$ winning the Gonality Game}    \label{figure:icosahedron_independent_set_divisor}
\end{figure}

Now let's find a lower bound on gonality.  The following lemma will be a useful tool for a number of arguments we will make; the proof is a little long, so we'll save it for Appendix \ref{appendix:icosahedron_outdegrees}.

\begin{lemma}\label{lemma:icosahedron_outdegrees}
    Let $H$ be a subgraph of $\icosahedron$.  If $|V(H)|=2$ or $|V(H)|=10$, then $\outdeg(H)\geq 8$.  If $3\leq |V(H)|\leq 9$, then $\outdeg(H)\geq 9$.
\end{lemma}

Let's consider a scramble $\mathcal{S}$ whose eggs each have two vertices, and let's include every such possible egg.  In other words, our scramble consists of all eggs $\{u,v\}$ where $u$ and $v$ are connected by an edge.  This is called the \emph{$2$-uniform scramble} in \cite{uniform_scrambles}.

\begin{proposition}
    We have $||\mathcal{S}||=8$.
\end{proposition}

\begin{proof}
    Let's star by computing $h(\mathcal{S})$.  First, pick an independent set of vertices of $V(G)$, and then let $W$ be the set of vertices in $V(G)$ that \emph{aren't} in that independent set.  Then, $W$ forms a hitting set: the vertices unhit by $W$ are independent and thus don't share edges, meaning that no egg is unhit.  Since $\alpha(\icosahedron)=3$, we can find a hitting set of size $9$.  In fact, for any hitting set of $\mathcal{S}$, the unhit vertices must form an independent set:  otherwise two unhit vertices would be adjacent, and would therefore form an unhit egg.  Since the largest independent sets have $3$ vertices, the smallest hitting sets have $9$ vertices.  Thus $h(\mathcal{S})=9$.

    For the egg-cut number, we are searching for the smallest number of edges to delete from the graph to split it into two pieces, each containing an edge.  These pieces must have between $2$ and $10$ vertices, and thus have outdegree at least $8$ by Lemma \ref{lemma:icosahedron_outdegrees}.  This means at least $8$ edges must be deleted for an egg-cut.  There do indeed exist egg-cuts with $8$ edges; for instance, choose two adjacent vertices, and delete all edges incident to each except for their shared edge.  Thus we have $e(\mathcal{S})=8$.

    In conclusion, we find
    \[||\mathcal{S}||=\min\{9,8\}=8.\]
\end{proof}

We now know $\gon(G)\leq 9$, and $\sn(G)\geq 8$.  From here we could either try to find a scramble of higher order, or a divisor of of degree $8$ that wins the Gonality Game.  We'll show the first is impossible, so we'll know there's no way to get a better lower bound using scramble number.  To do this we'll use the \emph{screewidth} of a graph; this is defined formally in \cite{screewidth-og}, though we'll present a quick definition here.

Say we're given a graph $G$.  Draw any tree $T$ (a connected graph with no cycles); we'll refer to its vertices as \emph{nodes} and edges as \emph{links}.  Draw the vertices of $G$ inside the nodes of $T$ (in any way you like), and then draw the edges of $G$ within $T$ in the simplest way possible, following the shortest path of nodes and links connecting the two vertices.  This drawing of $G$ within $T$ is called a \emph{tree-cut decomposition}.  You can see two examples of tree-cut decompositions of $\icosahedron$ in Figure \ref{figure:screewidth_icosahedron}.  The first has a tree with two nodes joined by a link; two adjacent vertices of the icosahedron are placed in one node, with the other $10$ vertices in the other node.  The second has a tree with three nodes, and can be obtained from the first tree-cut decomposition by ``pulling off'' another pair of adjacent vertices from the largest node.

\begin{figure}[hbt]
    \centering
    \includegraphics[scale=0.7]{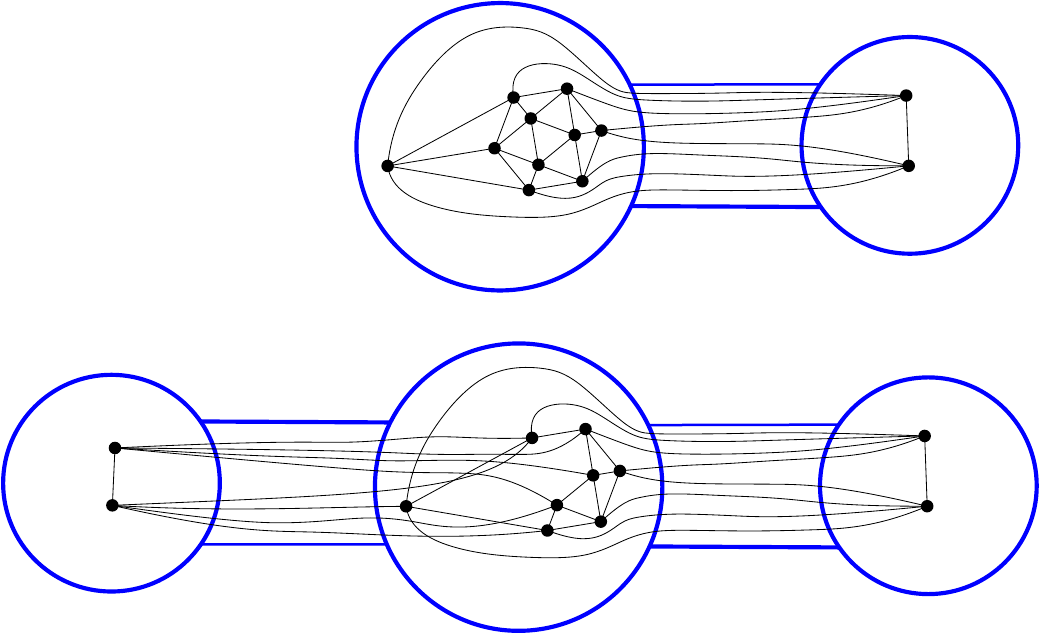}
    \caption{Two tree-cut decompositions of the icosahedron.  The first has width $10$, the second has width $8$.}
    \label{figure:screewidth_icosahedron}
\end{figure}

Once we have a tree-cut decomposition, we  compute a number called its \emph{width}.  For each link in $T$, count the number of $G$'s edges passing through it.  For each node in $T$, count the vertices from $G$ in it, and add the number of edges of $G$ that ``tunnel'' through the node, passing through it without having either endpoint in that node (the examples we've illustrated don't have any tunneling edges).  Among all these numbers (one for each link, and one for each node), take the \emph{maximum}.  This number is called the {width} of the tree-cut decomposition.  For our examples in Figure \ref{figure:screewidth_icosahedron}, you can check that the widths are $10$ and $8$, respectively.

Finally, the screewidth of a graph $G$, written $\scw(G)$, is the minimum width of any tree-cut decomposition of $G$.  From our examples, we know that $\scw(\icosahedron)\leq 8$.  Why is this useful information?  Well, the following theorem relates screewidth to scramble number.

\begin{theorem}[Theorem 1.1 in \cite{screewidth-og}] For any graph $G$, we have $\sn(G)\leq \scw(G)$
\end{theorem}

We now know that $8\leq \sn(\icosahedron)\leq \scw(\icosahedron)\leq 8$, so in fact $\sn(\icosahedron)=8$.  So, we know we can't possibly find a better scramble!  It's important to emphasize that screewidth did not tell us anything directly about gonality; rather, it told us what the limits were on scramble number, so we didn't need to spend more time searching for a higher order scramble to better lower bound gonality.  However, an interesting open question is the following:  do we have $\scw(G)\leq \gon(G)$ for all graphs $G$?  The answer is yes for all known examples, and if it does hold in general, then screewidth would become the strongest known lower bound on gonality!

Now that we have $\sn(\icosahedron)=8$, we might try to find a divisor of degree $8$ that wins the Gonality Game on the icosahedron.  But try as we might, whether by hand or with a computer, we can find no such divisor.  So, we should try to argue that the gonality is indeed $9$.

Here will be our strategy:
\begin{itemize}
    \item Let $D$ be any debt-free divisor of degree $8$.
    \item Cleverly choose a vertex $q$ where $D$ doesn't place a chip, and place $-1$ debt there.
    \item  Run Dhar's burning algorithm from $q$, and argue that the whole graph burns.
\end{itemize}
This will show that $D$ cannot win the Gonality Game, since it cannot eliminate the debt on $q$. Since $D$ was an arbitrary debt-free divisor with $8$ chips, we'll know that we need at least $9$ chips to win the Gonality Game.

One flaw in this strategy is that there \emph{are} divisors of  degree $D$ such that the whole graph won't burn immediately, no matter how we choose $q$.  For instance, place $8$ chips on a single vertex $v$; choosing any other vertex as $q$ leads to the graph besides $v$ burning, but as $v$ remains unburned we chip-fire it and must start our burning process again.  So we should choose $D$ more cleverly, using the following result.

\begin{theorem}[A corollary of Theorem 1.3 in \cite{backman2017riemannroch}, discussed in \cite{aidun2019gonality}] \label{theorem:spread_chips_out}
 Let $D$ be a divisor with no debt on its vertices, with $\deg(D)\leq |E(G)|-|V(G)|$.  Then $D$ is equivalent to another divisor $D'$ with no debt such that
 \begin{itemize}
     \item for any vertex $v$, $D$ places at most $\val(v)-1$ chips on it; and
     \item for any two adjacent vertices $v$ and $w$, $D$ does not place both $\val(v)-1$ chips on $v$ and $\val(w)-1$ chips on $w$.
 \end{itemize}
\end{theorem}

Intuitively this says that as long as we don't have too many chips on our graph, we can ``spread them out'' pretty well.  This makes the whole graph much more susceptible to burning in Dhar's burning algorithm, as there is a limit to the size of piles of chips to defend against fires.

We're ready to prove the following.

\begin{theorem}
    The gonality of the icosahedron graph is $9$.
\end{theorem}

\begin{proof}  We have already seen that $\gon(\icosahedron)\leq 9$.

For our icosahedron graph we have $|E(\icosahedron)|-|V(\icosahedron)|=30-12=18$, so certainly a debt-free divisor of degree $8$ satisfies the the assumptions of Theorem \ref{theorem:spread_chips_out}.  Let $D$ be any debt-free divisor of degree $8$ on $\icosahedron$; perhaps replacing it with an equivalent debt-free divisor, we can assume that $D$ places at most $4$ chips on any vertex, and that it does not place $4$ chips on each of a pair of adjacent vertices.  We wish to show that $D$ does \emph{not} win the Gonality Game.
 
The next step is to carefully choose a vertex $q$ to run Dhar's burning algorithm from.  Since $D$ has degree $8$, at least four vertices are unchipped; because $\alpha(\icosahedron)=3$, these four vertices cannot form an independent set, so choose $q$ with no chips so that it has at least one neighbor with no chips.  Place $-1$ debt on $q$, and run Dhar's burning algorithm. If the whole graph burns, then $D$ cannot eliminate the debt on $q$ and we are done, so we only have to deal with the possibility that the burning process finishes with a set $B$ of burned vertices with $|B|\leq 11$.  We will argue that the burning process cannot stop at any such point, or that if it can then some other vertex $q'$ can demonstrate that $D$ fails to win the Gonality Game.

We cannot have $|B|=1$, since $q$ and its unchipped neighbor both burn.  It is possible for $|B|=2$, although only barely:  since $q$ and its unchipped neighbor together have eight edges connected them to other vertices, the eight chips of $D$ would have to be placed on the other endpoints of those edges to contain the fire.  However, in this case we know exactly what our divisor $D$ is: it would have to be as illustrated in Figure \ref{figure:two_burn_stop_icosahedron}, where $v$ is the unchipped neighbor of $q$. (By the symmetry of the graph, it does not matter which pair of adjacent vertices are $q$ and $v$.)  If our divisor $D$ has this form, place  $-1$ debt on the vertex $q'$ as illustrated, rather than on $q$. Note that the whole graph will burn running Dhar's burning algorithm from $q'$, meaning $D$ does not win the Gonality Game.

\begin{figure}[hbt]
    \centering
    \includegraphics{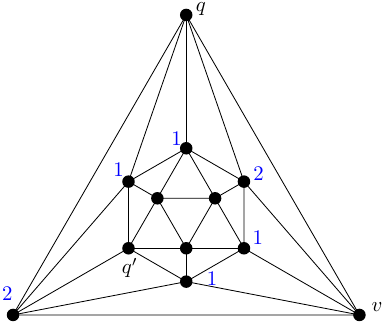}
    \caption{The required form of $D$ if the fire from $q$ burns only $q$ and $v$.  Although $D$ can eliminate $-1$ debt added to $q$, it cannot eliminate $-1$ debt added to $q'$.}
    \label{figure:two_burn_stop_icosahedron}
\end{figure}

We are left to rule out the possibility that $|B|\geq 3$.  If $3\leq |B|\leq 9$, then by Lemma \ref{lemma:icosahedron_outdegrees} there are $9$ (or more) edges connecting the vertices of $U$ to the rest of the graph; but to contain the fire would require $9$ (or more) chips, so the fire cannot stabilize.  If $|B|=10$, let $u$ and $u'$ be the unburned vertices.  If $u$ and $u'$ are not adjacent, then each has $5$ burning edges and at most $4$ chips, and so will burn.  If they are adjacent, then each has $4$ burning edges.  One has at most $3$ chips and so will burn, at which point the other has $5$ burning edges and at most $4$ chips and will burn as well. Finally, if $|B|=11$, let $u$ be the unburned vertex.  It has at most $4$ chips, but $5$ burning edges, and so burns as well.

Thus we find that the whole graph burns when running Dhar's burning algorithm from some unchipped vertex (either $q$, or $q'$ in the case that the fire from $q$ stabilized early), showing that $D$ does not win the Gonality Game.  We can therefore conclude that $\gon(\icosahedron)\geq 9$, finishing the proof.
\end{proof}

We've now successfully computed the gonality of our five Platonic graphs!  The results of these computations, along with the strategies we used, are summarized in Table \ref{table:summary}.

\begin{table}[hbt]
\begin{tabular}{|c|c|c|c|}
\hline
\textbf{Graph name} & \textbf{Gonality} & \textbf{Lower bound technique(s)} & \textbf{Upper bound technique(s)} \\ \hline
Tetrahedron & $3$ & Dhargument, $\tw(G)$, $\delta(G)$ & Independence number \\ \hline
Octahedron & $4$ & $\tw(G)$, $\delta(G)$ & Independence number \\ \hline
Cube & $4$ & Scramble number & \begin{tabular}[c]{@{}c@{}}Independence number,\\ Cartesian product\end{tabular} \\ \hline
Dodecahedron & $6$ & Scramble number, $\tw(G)$ & Custom divisor \\ \hline
Icosahedron & $9$ & Dhargument & Independence number \\ \hline
\end{tabular}
\caption{A summary of our results and techniques}
\label{table:summary}
\end{table}

\begin{exploringfurther}[\textbf{The Archimedean solids and the Platonic solids in higher dimensions}]
Now that we've handled the Platonic solid graphs, I'm sure you're excited for more graphs to compute the gonalites of!  There are all sorts of graph families out there with unknown graph gonalities, but one family of similar flavor could be the graphs coming from the \emph{Archimedean solids}.  These are a slight generalization from the Platonic solids, where every face must still be a regular polygon and every vertex must look the same, but now more than one regular polygon is used.  One of the most famous such solids is the \emph{truncated icosahedron}, or \emph{soccer ball}, pictured in Figure \ref{figure:truncated_icosahedron} along with its graph.  Its faces are regular hexagons and pentagons, and it is obtained by cutting off the twenty vertices of an icosahedron. 

\begin{figure}[hbt]
    \centering
    \includegraphics[scale=0.4]{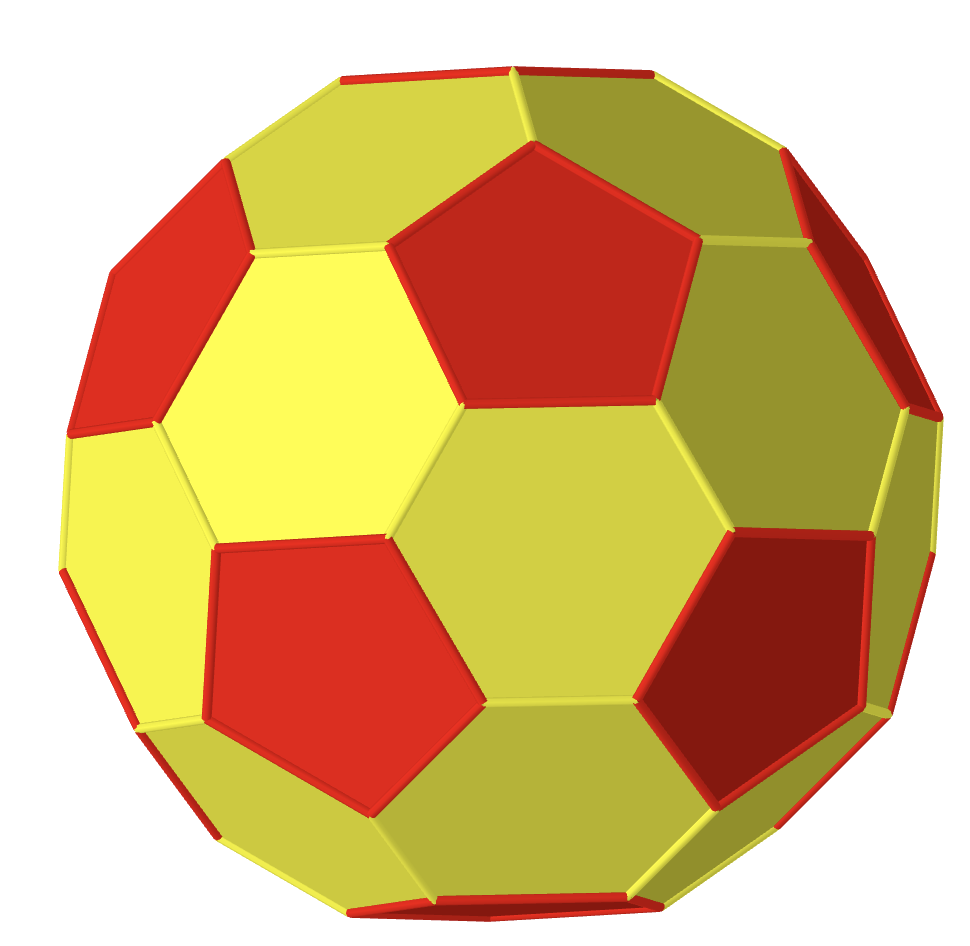}\quad\quad\includegraphics[scale=0.45]{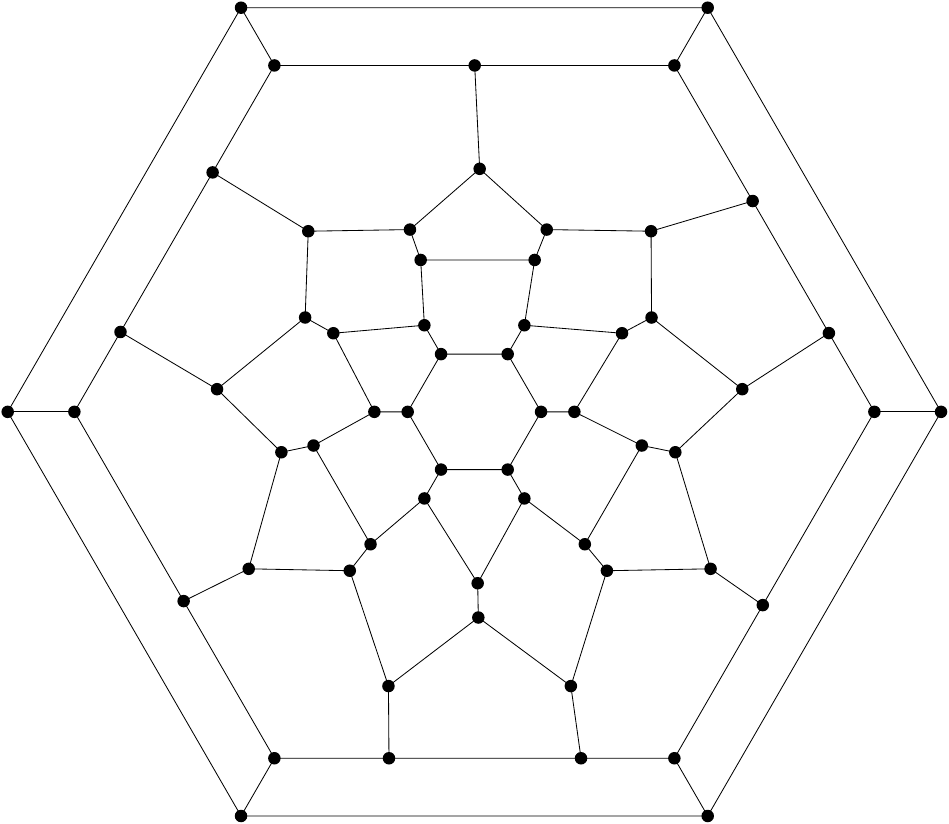}
    \caption{The truncated icosahedron, and its graph}
    \label{figure:truncated_icosahedron}
\end{figure}
All told, there are $13$ Archimedean solids (plus two infinite families, the \emph{prisms} and the \emph{antiprisms}).  That's a great set of graphs to test out our techniques.

Another direction for generalization could be moving from three-dimensional Platonic solids to regular \emph{polytopes} (analogs of polyhedra) in higher dimensions.  It turns out in five dimensions and up, there are exactly three regular \emph{polytopes}:  the $n$-dimensional simplex (a generalization of the tetrahedron), the $n$-dimensional orthoplex (a generalization of the octahedron), and the $n$-dimensional hypercube.  The graph of the $n$-dimensional simplex is always a complete graph on $n-1$ vertices, and the $n$-dimensional orthoplex gives us the complete multipartite graph $K_{2,2,\ldots,2}$ (with $n$ copies of $2$), so those are handled by our sections on the tetrahedron and the octahedron.  We've already discussed the $n$-dimensional hypercube graph $Q_n$, whose gonality is unknown for $n\geq 6$.

When $n=4$, there are three more regular polytopes: the $24$-cell, the $600$-cell, and the $120$-cell, which have respectively $24$, $120$, and $600$ vertices.  The gonality of the $24$-cell graph could probably be computed by brute force, but if you can compute the gonality of either of the other two, it would be very impressive!
\end{exploringfurther}

\section{Further reading}

Hopefully you're now excited to compute graph gonalities, and to study some related topics!  Here are some topics and readings you might enjoy.

\begin{itemize}
    \item As a general chip-firing resource, the textbooks \cite{sandpiles} and \cite{klivans} are very helpful.  The first five chapters of \cite{sandpiles} focuses on divisor theory on graphs (including a proof of the Riemann-Roch Theorem for Graphs), with selected topics in later chapter.  

    \item If you'd like to see more techniques for computing the gonality of families of graphs, you can find lots of those:  product graphs \cite{aidun2019gonality,echavarria2021scramble}; rook's graphs \cite{rooks_gonality}, queen's graphs \cite{queens_graph_gonality}, and other graphs coming from chess \cite{gonality_chess_graphs}; random graphs \cite{gonality_of_random_graphs}; graphs with universal vertices \cite{sparse}; and Ferrers rooks graphs \cite{ferrers_rooks_graphs}.

    \item In the Gonality Game, Player B got to place $-1$ chips on the graph.  What if they get to place $-r$ chips on the graph instead (for some $r\geq 1$), spread out however they like?  The minimum number of chips Player A needs to win in that game is called the \emph{$r^{th}$ gonality of $G$}, written $\gon_r(G)$.  (In the language introduced around the Riemann-Roch Theorem for graphs, $\gon_r(G)$ is the minimum degree of a divisor of rank at least $r$ on $G$.

    Much less is known about higher gonalities than first gonality.  A lot is known if $g(G)$ is small, as explored in \cite{gonseq}, and there are a few infinite families where all higher gonalities are known, like complete graphs and complete bipartite graphs.  Beyond that, lots is open!  For instance, what are the higher gonalities of the Platonic solids?

    \item Our chip-firing games have been played on combinatorial graphs, consisting of vertices and edges.  Chip-firing games can also be played on \emph{metric graphs}, geometric objects obtained by assigning lengths to each edge.  Here divisors can place chips on the middle of edges, and chip-firing is defined to allow for continuous movement of chips along edges.  A great introduction to this theory can be found in \cite{yoav_metric}.

    The interplay between gonality on combinatorial and on metric graphs is subtle, and only recently was it made clear in \cite{discrete_metric_different} how metric gonality can be computed through combinatorial graphs.  That paper has several great examples of how metric graphs might need fewer chips to win the Gonality Game than seemingly identical combinatorial graphs, and poses many open questions to explore.

    \item In case you've taken an Abstract Algebra or a Group Theory course:  the set of all divisors on a graph, up to chip-firing equivalence, forms a group!  How can we compute this group?  What structure can this group have? A great introduction to this topic, as well as open research problems on it, can be found in \cite{chip_firing_games_critical_groups} (in addition to being published as a chapter in a textbook, that write-up is also available on \url{arxiv.org}). 
\end{itemize}

\bibliographystyle{plain}

\appendix

\section{Outdegrees of subgraphs of the icosahedron.}
\label{appendix:icosahedron_outdegrees}

Here we provide a proof of Lemma \ref{lemma:icosahedron_outdegrees}, which says that for an induced subgraph $H$ of the icosahedron, we have $\outdeg(H)\geq 8$ if $|V(H)|=2$ or $|V(H)|=10$, and $\outdeg(H)\geq 9$ if $3\leq |V(H)|\leq 9$.

\begin{proof}
    
First note the outdegree of $H$ is the same as the outdegree of $G-H$, the graph obtained by deleting $H$ from $G$.  Since our claim is symmetric when we replace $|V(H)|$ with $12-|V(H)|$, we can assume $H$ has at most six vertices.

Similar to the dodecahedron, we know that
\[\outdeg(H)=5|V(H)|-2|E(H)|.\]
Assume $|V(H)|=2$.  The two vertices in $H$ each have valence at least $5$, and share at most one edge in $G$.  Thus we have $\outdeg(H)\geq 5\cdot 2-2\cdot 1=8$, as desired.

Now we consider the rage $3\leq |V(H)|\leq 6$, splitting into cases.

\begin{itemize}
    \item $|V(H)|=3$:  if $H$ has three vertices, then it has at most three edges, which occurs if the three vertices form a copy of $K_3$.  Thus $\outdeg(H)\geq 5\cdot 3-2\cdot 3=9$.
    \item $|V(H)|=4$:  If $H$ has four vertices, then at first glance it has at most $6$ edges, which occurs only if the three vertices form a copy of $K_4$.  However, there is no $K_4$ subgraph of $\icosahedron$.  Thus there are at most $5$ edges, and so $\outdeg(H)\geq 5\cdot 4-2\cdot 5=10$.
    
    \item $|V(H)|=5$: Here we will argue $V(H)$ has at most $7$ edges.  Suppose there are $8$ (or more) edges.  Let $V(H)=\{u_1,\ldots,u_5\}$, where without loss of generality $u_1$ has the smallest valence in $H$. First note that none of the five vertices can share an edge with the other four; a vertex in $\icosahedron$ with its five neighbors yields a graph with $9$ edges, and removing one of those vertices removes at least $2$ edges. Since $u_2,\ldots,u_5$ share at most $5$ edges, $u_1$ must be connected to at least three of them.  It follows that each vertex has valence $3$ in $H$ (as valence $4$ has been ruled out).  So, $H$ is a graph on $5$ vertices where every vertex has valence $3$. But this is a contradiction, since three and five are both odd; for instance, since adding up valences double-counts the number of edges, it would imply that $H$ has $\frac{3\cdot 5}{2}=7.5$ edges, which is impossible.  Thus we have at most $7$ edges in $H$ as desired.  Thus we have $\outdeg(H)\geq 5\cdot 5-2\cdot 7=11$.
    
    \item $|V(H)|=6$:  we will argue that $H$ has most $10$ edges.  Suppose it has $11$ (or more) edges.  Let $V(H)=\{u_1,\ldots,u_6\}$, where without loss of generality $u_1$ has the smallest valence in $H$.  First note that no vertex is incident to all five of the others: a vertex in $\icosahedron$ together with its five neighbors yields only $9$ internal edges.  Since $u_2,\ldots,u_6$ can have at most $7$ edges between them, we know that $u_1$ must be incident to four of them.  It follows that each of the six vertices has valence $4$ in $H$ (as valence $5$ has been ruled out), so that the total number of edges is $\frac{6\cdot 4}{2}=12$.  Deleting $u_1$ leaves us with $8$ edges among the five vertices $u_2,\ldots,u_6$, which is impossible as argued in the previous case.  Thus there are at most $10$ edges in $H$, implying that $\outdeg(H)\geq 6\cdot 5-2\cdot 10=10$.
\end{itemize}
\end{proof}

\end{document}